\documentclass[11pt]{amsart}
\usepackage{amsmath, amsfonts, amsthm, amssymb}
\usepackage[normalem]{ulem}
\allowdisplaybreaks[3]
\usepackage{graphicx}
\usepackage[margin=1.75in, top=1.25in, bottom=1.25in]{geometry}
\usepackage{multirow}

\numberwithin{equation}{section}
\usepackage[shortlabels]{enumitem}
\setlist[itemize]{itemsep=0cm}
\setlist[enumerate]{itemsep=0cm}

\usepackage[svgnames]{xcolor}
\usepackage[colorlinks=true, allcolors=black, urlcolor=DarkBlue]{hyperref}

\usepackage{booktabs}

\theoremstyle{plain}
\newtheorem{thm}{Theorem}[section]
\newtheorem{prop}[thm]{Proposition}
\newtheorem{lem}[thm]{Lemma}
\newtheorem{cor}[thm]{Corollary}

\theoremstyle{definition}

\theoremstyle{remark}
\newtheorem{ex}[thm]{Example}
\newtheorem{remark}[thm]{Remark}

\newcommand{\defnterm}[1]{\emph{#1}}

\newcommand{\RR}{\mathbb{R}}

\newcommand{\ZZ}{\mathbb{Z}}

\newcommand{\cV}{\mathcal{V}}

\renewcommand{\bar}[1]{\overline{#1}}

\renewcommand{\tilde}[1]{\widetilde{#1}}

\renewcommand{\phi}{\varphi}

\newcommand{\D}{\Delta}

\let \1 \relax
\DeclareMathOperator{\1}{\textbf{1}}

\newcommand{\supp}{\operatorname{supp}}

\newcommand{\qqt}[1]{\qquad \text{#1} \qquad}
\newcommand{\qt}[1]{\quad \text{#1} \quad}

\newcommand{\qqand}{\qqt{and}}
\newcommand{\qand}{\qt{and}}

\newcommand{\tablespacing}{\vspace{0.2cm}}

\usepackage[foot]{amsaddr}

\title{Geometry and Laplacian on Discrete Magic Carpets}
\author{Eric Goodman$^1$}
\address{$^1$The Department of Mathematics, The University of Pennsylvania \\ David Rittenhouse Lab 3C9, 209 South 33rd Street, Philadelphia, PA 19104}
\email{ericgood@sas.upenn.edu}

\author{Chun-Yin Siu$^2$}
\address{$^2$The Department of Mathematics, The Chinese University of Hong Kong \\ 222A Lady Shaw Bldg, The Chinese University of Hong Kong, Hong Kong}
\email{alexcysiu@gmail.com}

\author{Robert S. Strichartz$^3$}
\address{$^3$The Department of Mathematics, Cornell University \\ 563 Malott Hall, Cornell University, Ithaca, NY USA 14853}
\email{str@math.cornell.edu}

\date{\today}

\begin{document}

\maketitle


\begin{abstract}
    We study several variants of the classical Sierpinski Carpet (SC) fractal. The main examples we call infinite magic carpets (IMC), obtained by taking an infinite blowup of a discrete graph approximation to SC and identifying edges using torus, Klein bottle or projective plane type identifications. We use both theoretical and experimental methods. We prove estimates for the size of metric balls that are close to optimal. We obtain numerical approximations to the spectrum of the graph Laplacian on IMC and to solutions of the associated differential equations: Laplace equation, heat equation and wave equation. We present evidence that the random walk on IMC is transient, and that the full spectral resolution of the Laplacian on IMC involves only continuous spectrum. This paper is a contribution to a general program of eliminating unwanted boundaries in the theory of analysis on fractals.
\end{abstract}

\keywords{\small{
Subject Classification: 2010 AMS Subject Classification: Primary 28A80.\\
Keywords: 
fractal, Sierpinski carpet, magic carpets, metric balls, random walk, Laplacian, Heat kernel, wave propagator, harmonic functions, Weyl ratio
}}

\section{Introduction}
The Sierpinski Carpet ($SC$) is a classical self-similar fractal generated by an iterated function system of eight contractive similarities in the plane with contraction ratio 1/3. Figure~\ref{fig:intro-SC-zoom-in} shows the first three iterations of the approximation to $SC$ obtained from the unit square by applying the contractions.
\begin{figure}[h]
	\centering
	\includegraphics[width=2cm]{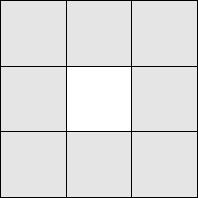}
	\hspace{0.5cm}
	\includegraphics[width=2cm]{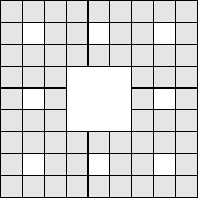}
	\hspace{0.5cm}
	\includegraphics[width=2cm]{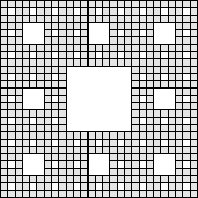}
	\caption{Approximations to $SC$}
	\label{fig:intro-SC-zoom-in}
\end{figure}
Two constructions of a Brownian motion on $SC$ were given by Barlow and Bass \cite{barlow-bass} and Kusuoka and Zhou \cite{kusuoka-zhou}, and these give rise to a symmetric, self-similar energy (Dirichlet form) and Laplacian. Only recently, in \cite{barlow-bass-kumagai-teplyaev} has it been shown that there is, up to a constant multiple, a unique symmetric self-similar Laplacian on $SC$, so the two constructions are equivalent, and also certain passages to subsequences in the constructions are unnecessary. Although the Brownian motion approach yields sub-Gaussian heat kernel estimates, it does not yield detailed information about the eigenvalues and eigenfunctions of the Laplacian (with appropriate boundary conditions). Nevertheless, several experimental approaches have yielded good numerical approximations to the spectrum \cite{begue-kalloniatis-strichartz}.

One rather vexing question concerns the nature of the analytic boundary of $SC$. (Note that there is no meaningful notion of topological boundary, since $SC$ has no interior.) This is usually taken to be the boundary of the square containing $SC$. But a glance at Figure~\ref{fig:intro-SC-zoom-in} shows that there are infinitely many line segments in $SC$ that are locally isometric to portions of this boundary, so the standard choice appears somewhat arbitrary and capricious. In an attempt to get rid of the boundary altogether, a related fractal called the Magic Carpet ($MC$) was introduced in~\cite{bello-li-strichartz} and further studied in~\cite{molitor-ott-strichartz} where potential boundary line segments are identified. Thus the opposite sides of the original square are identified with the same orientation to produce a torus. At stage $m$ of the approximation, $8^{m-1}$ vacant squares are cut into the previous approximation, and again the opposite sides are identified with the same orientation. This yields a set of $8^m$ squares (called $m$-cells) of side length $1/3^m$, and each square has exactly four neighboring squares on the top, bottom, left, and right. We call this cell graph $MC_m$. See Figure~\ref{fig:intro-mc} for an illustration of the $m = 1$ and $m = 2$ approximations. Of course these approximations do not embed in the plane. They should be thought of as surfaces that are flat except for singular points at the corners of identified edges.
\begin{figure}
	\centering
	\includegraphics[width=3cm]{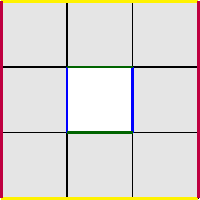}
	\hspace{0.5cm}
	\includegraphics[width=3cm]{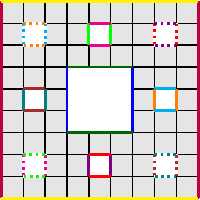}
	\caption{Approximations to $MC$}
	\label{fig:intro-mc}
\end{figure}

In this paper we will denote the above magic carpet $MCT$ to indicate that we have made torus-type identifications of edges. We will also consider $MCK$ and $MCP$, where we make Klein bottle or projective plane identifications, as shown in Figure~\ref{fig:ident-types}. (Note that we could make horizontal or vertical Klein bottle identifications, denoted $K_H$ and $K_V$, but in this uniform case the two fractals are isometric.) Later, in Section~\ref{sec: homogeneous}, we will consider still other fractals of homogeneous type, where we make one of the four identification types---$T, P, K_H$, or $K_V$---on each level.
\begin{figure}
	\centering
	\includegraphics{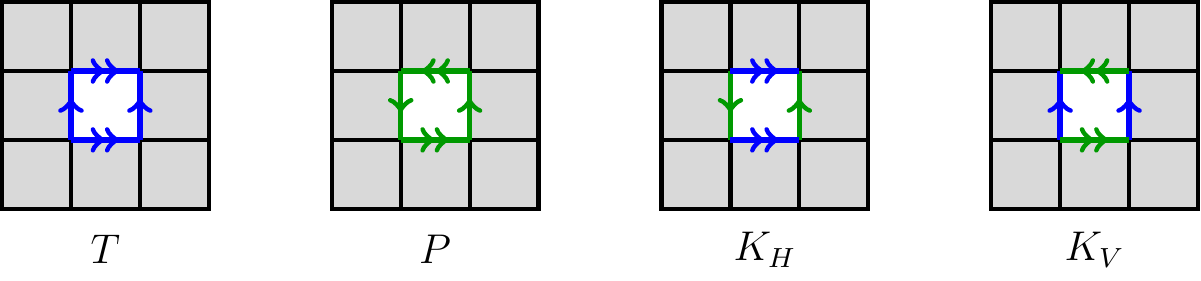}
	\caption{Identification Types}
	\label{fig:ident-types}
\end{figure}

We will also consider infinite graphs obtained by blowing up the approximations. In other words, take the level-$m$ approximation to $MC$ and regard each $m$-cell as the vertex of a graph $MC_m$, and then take the appropriate limit as $m \to \infty$ to obtain the infinite magic carpet graph $IMC$. More precisely, we write $\widetilde{MC}_m$ to be the cell graph of $MC_m$ without identifying the boundary of the outer square. Then we have embeddings
\[ \widetilde{MC}_1 \subseteq \widetilde{MC}_2 \subseteq \cdots \]
and $IMC$ is simply the union. Note that there are many different embedding choices (eight on each level) so $IMC$ is not unique. In the generic case where there is no boundary it turns out not to matter for what we do here. In the future there may be questions that have different answers depending on these choices. Note that $IMC$ is an infinite, 4-regular graph.

We begin by investigating the geometry of the graph $IMC$. For each fixed vertex, $x$, let $B(x, r)$ denote the ball of radius $r$ in the geodesic graph metric. What is the cardinality $\# B(x, r)$ as $r \to \infty$? In Section~\ref{sec: cardinality} we prove that $\# B(x, r) = O\big(r^3\big)$. More precisely, for $IMCT$ and $IMCK$ we have upper and lower bounds of a constant times $r^3$. For $IMCP$ we obtain the same type of upper bound, but our best lower bound is $c \left( r / \log r \right)^3$. We also undertake some numerical experiments that suggest $\lim_{r \to \infty} \# B(x, r) / r^3$ exists for torus and Klein bottle identifications. If this is true for one $x$ then the same limit holds for all $x$.

In Section~\ref{sec: random walks} we examine random walks on $IMC$. The main question is to decide whether these are recurrent or transient. We gather numerical data on two fronts:
	(i)~we compute the percentage of walks that return to the starting point as the length of the walk varies, and
	(ii)~we compute the effective resistance from a fixed point to the boundary of a large square (which should tend to infinity if the walk is recurrent and remain bounded if the walk is transient, \textit{cf.}~\cite{doyle-snell}, \cite{woess}). Neither test is decisive, but we present the data.

In Section~\ref{sec: spectrum on IMC} we study the spectrum of the graph Laplacian, $-\D u(x) = u(x) - \frac{1}{4} \sum_{y \sim x} u(y)$, on $IMC$. Here the main question is whether the spectrum is pure point, continuous, or a mixture of the two. In order to explore the possibility of square-summable eigenfunctions (point spectrum), we numerically solve the Dirichlet problem $-\D u = \lambda u$ inside $\tilde{MC}_m$ for large $m$ with $u \equiv 0$ on the boundary. If $u$ were a square-summable eigenfunction, then eventually it would be very close to zero on a large neighborhood of the boundary of $\tilde{MC}_m$, and so it would be very close to one of the Dirichlet eigenfunctions, and this Dirichlet eigenfunction would vanish rapidly as you approach the boundary. We do not see any such Dirichlet eigenfunctions, so this provides strong numerical evidence that the spectrum is continuous.

This also tells us that the Dirichlet spectrum is unrelated to the spectrum of $IMC$. Nevertheless, as we will see later, the data is not completely useless, as it allows us to study the heat kernel.

In Section~\ref{sec: heat kernel} we study the heat kernel on $IMC$ by approximating it by the Dirichlet heat kernel on $\widetilde{MC}_m$, given by
\begin{align}
	H_t^{(m)}(x, y) = \sum_i e^{-\lambda_i t} u_i(x) u_i(y),
\end{align}
where $\{ u_i \}$ is an orthonormal basis of Dirichlet eigenfunctions on $\widetilde{MC}_m$ with eigenvalues $\lambda_i$ (all of these depend on $m$, of course, but we prefer not to burden the notation to make this explicit). Since the heat kernel is highly localized, if $x$ and $y$ are not too close to the boundary the choice of Dirichlet boundary conditions should have only a negligible effect on the heat kernel. The two fundamental questions here concern the on-diagonal behavior and the off-diagonal behavior. For $x = y$, we ask if there is some power law behavior
\begin{align}
	\label{eq:heat-kernel-simple-power-law}
	H_t(x, x) = O\left(t^{-\beta} \right) \qqt{as} t \to \infty
\end{align}
for some $\beta$. Surprisingly, our data suggest that instead of~\eqref{eq:heat-kernel-simple-power-law}, it is more likely that
\begin{align}
	H_t(x, x) = O\left(t^{-\beta(x)} \right) \qqt{as} t \to \infty
\end{align}
where $\beta(x)$ depends on $x$. We present data to support this. The $\beta(x)$ values for $x$ far from the boundary satisfy $\beta(x) > 1$. For the off-diagonal behavior, we fix $y$ and $t$ and examine the decay of $H_t(x,y)$ as $x$ moves away from $y$. We present two types of data:
	(i)~the graph of $H_t(x,y)$ as a function of $x$ as $x$ varies along a line segment in $IMC$ containing $y$, and
	(ii)~a scatter plot of the values of $H_t(x,y)$, where $x$ varies over all points of distance $r$ to $y$, with $r$ varying.
Because these values of the heat kernel are close to zero, the graphs of $\log H_t(x, y)$ reveal more information.

In Section~\ref{sec: wave} we study the wave propagator
\begin{align}
	\label{eq:wave-propagator}
	W_t^{(m)}(x,y) = \sum_i \frac{\sin \lambda_i t}{\lambda_i} u_i(x) u_i(y)
\end{align}
that provides the solution
\[ u(x, t) = \sum_y \left( W_t^{(m)}(x,y) \, g(y) + \left.\frac{d W_t^{(m)}}{dt}\right|_{(x,y)} f(y) \right) \]
to the wave equation
\[ \frac{\partial^2 u}{\partial t^2} = \D^{(m)} u \]
(where $\D^{(m)}$ denotes the Laplacian on $\tilde{MC}_m$) with initial conditions
\[ u(x, 0) = f(x) \qqand \frac{\partial u}{\partial t}(x, 0) = g(x). \]
Although the wave propagator is not as localized as the heat kernel, it is still relatively localized for small $t$, so that our approximations give interesting information about wave propagation on the $IMC$ graph.

In Section~\ref{sec: harmonic} we study harmonic functions and the analog of the Poisson kernel on $\tilde{MC}_m$ obtained by setting boundary values $P(x, y) = \delta_{xy}$ for $y$ a fixed point, $x$ a variable point on the boundary, and $\D_x^{(m)} P(x,y) = 0$ for $x$ in the interior. The Poisson kernel decays as $x$ moves away from $y$, as seen in the graphs and scatter plots, but not as rapidly as the heat kernel. An interesting question that we have not been able to deal with is whether or not there is an analog of Liouville's Theorem: are bounded harmonic functions necessarily constant?

In Section~\ref{sec: uniform} we return to the finite fractal setting. By considering $MC_m$ as consisting of cells of size $1/3^m$ we obtain approximations to the original magic carpet $MCT$. We also do the same for the other identification types to obtain $MCP$ and $MCK$. We find convergence of eigenvalues as we vary $m$ with respect to a Laplacian renormalization factor $R$ that is slightly larger than~6 (it varies slightly with the identification type), and we can observe the refinement of eigenfunctions as $m$ increases by using an averaging process to pass from functions on level $m+1$ to functions on level $m$. We also see miniaturization of eigenfunctions that produce periodic eigenfunctions that are translates on copies of $V_m$ of a certain size. These may also be interpreted as periodic eigenfunctions on $IMC$, analogous to the functions $\cos \lambda k$ (for rational $\lambda$) on the integers.

We then compute the eigenvalue counting function,
\[ N(t) = \# \{ \lambda_i \; : \; \lambda_i \le t \}, \]
and the Weyl ratio,
\[ W(t) = \frac{N(t)}{t^\alpha} \qqt{for} \alpha = \frac{\log 8}{\log R}. \]
(Here, 1/8 is the renormalization factor for the standard measure on $MC$.) Note that $\alpha > 1$. This is quite different than for $SC$ \cite{begue-kalloniatis-strichartz}. The three different identification types yield qualitatively different Weyl ratio graphs. We do not see any periodicity in the log-log plots of $W(t)$.

In Section~\ref{sec: homogeneous} we continue in the finite fractal setting, but we allow the identification type to change from level to level. We call these \defnterm{homogeneous} magic carpets. Now there are actually four possibilities since the two Klein bottle identifications---horizontal, $K_H$, and vertical, $K_V$---are not always interchangeable if used on different levels. Thus we write $T, P, K_H, K_V, T$ for $m=4$ identifications, where the first $T$ means identify the outer boundary by $T$, the sides of the single large vacant square by $P$, the next eight largest vacant squares by $K_H$, and so on. We note that the final identification on the smallest level does not influence the spectrum, but of course it would lead to different fractals if we continue in the limit. This leads on level~4 to $256 = 4^4$ possible spectra, some of which are equivalent. Here we look at a representative sample, including all sixteen involving just $T$ and $P$. All cases are included in the website~\cite{website}. The question we would like to answer is the following: is there a qualitative procedure to use the Weyl ratio to deduce the particular identifications chosen? An idea proposed in \cite{drenning-strichartz} called spectral segmentation is that different segments of the spectrum relate to the identification types at different levels. While we are unable to support this hypothesis in full generality, we are able to see the signature of the first $k$ choices in the beginning segments of increasing length for $k = 1, 2, 3$.

In Section \ref{sec: spectral resolution} we discuss the experimental evidence that the spectral resolution of the Laplacian on $IMC$ is purely continuous.

Many of the ideas discussed in this paper are still conjectural, but we present a lot of data in figures and tables to support these conjectures experimentally. The website \cite{website} contains much more data. For the general theory of Laplacian on fractals the reader may consult the books \cite{Barlow98}, \cite{Kigami01} and \cite{strichartz-book}.


\section{Cardinality}
\label{sec: cardinality}

The distance between cells $x$ and $y$, $d(x,y)$, in the identified magic carpet blowup, $IMC$, is defined to be the length of the shortest path through cells from $x$ to $y$. Recall that we denote the ball of radius $r$ around a cell $x$ by
\[ B(x, r) = \{ y \; : \; \text{$y$ is a cell of $IMC$ with } d(x, y) < r \}, \]
whose cardinality will be denoted $\# B(x, r)$. The level-$m$ approximation to $IMC$, including the inner identifications, is denoted by $V_m$. An identification that is done at level $m + 1$ will be called an \defnterm{$m$-stitch}. With this notation, $V_{m+1}$ contains eight copies of $V_m$ and one $m$-stitch.


\subsection{The Lower Bound for Torus and Klein Identifications}
\label{sec:cardinality-torus-klein}

In this section we derive the $r^3$ lower bound for torus and Klein identifications. In each case, the argument is the same. Projective identifications are deferred to Section~\ref{sec:cardinality-projective-lower-bound}.

Consider first torus identifications. The left image of Figure~\ref{fig: corner-to-corner paths} shows a path of length~10 across $V_2$. To find a path across $V_3$, we may duplicate the path across $V_2$ and add six steps through the five red cells. We obtain the path of length~26 in the right image. To find a path across any higher $V_{m+1}$, we may repeat this process: duplicate the path across $V_m$ and add six steps through cells positioned as the red cells are. Letting $r_m$ be the length of this path across $V_m$, these lengths satisfy
\[ r_{m+1} = 2 r_m + 6 \qqt{with} r_2 = 10, \]
which has solution
\begin{align}
	\label{eq: r_m closed form}
	r_m = 2^{m+2} - 6.
\end{align}

\begin{figure}[h]
	\centering
	
	\begin{minipage}{1.5cm}
		\centering
		\includegraphics[width=1.5cm]{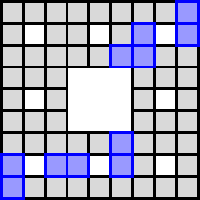} \\
		$V_2$
	\end{minipage}
	\hspace{1cm}
	\begin{minipage}{4.5cm}
		\centering
		\includegraphics[width=4.5cm]{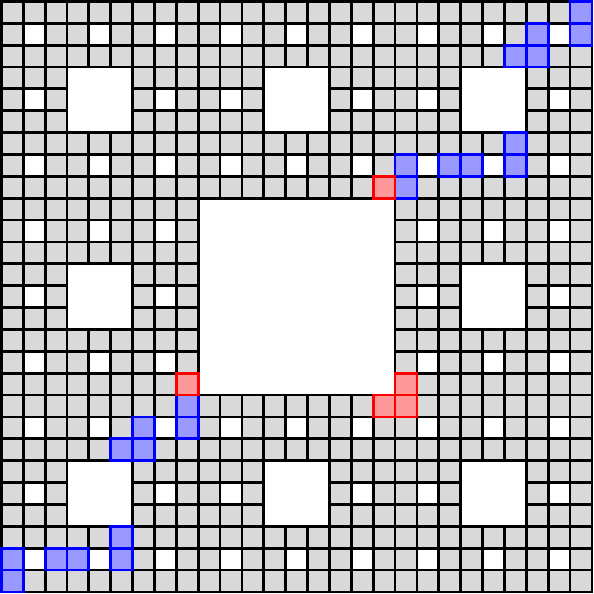} \\
		$V_3$
	\end{minipage}
	
	\caption{Paths of lengths~$r_2 = 10$ and~$r_3 = 26$ with torus identifications.}
	\label{fig: corner-to-corner paths}
\end{figure}

\begin{lem}
	\label{lem: can reach entire V_m in r_m steps}
	Given $r_m$ steps, a path from a corner cell of $V_m$ can reach any other cell in $V_m$.
\end{lem}
\begin{proof}
First, consider torus identifications. We induct on $m$ with base case $V_2$. We check this case by hand on the left of Figure~\ref{fig:distances-in-V_2}; indeed, all cells may be reached within $r_2$ steps.

\begin{figure}[h]
	\centering
	\makebox[\linewidth]{
		\includegraphics[width=5cm]{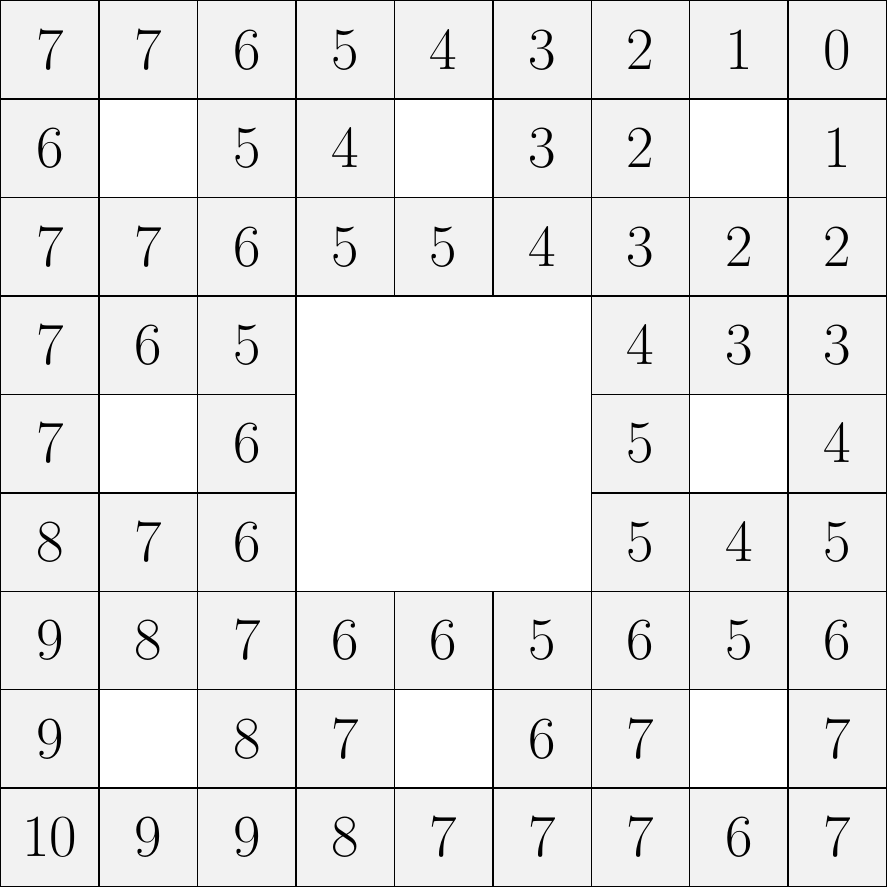}
		\hspace{0.5cm}
		\includegraphics[width=5cm]{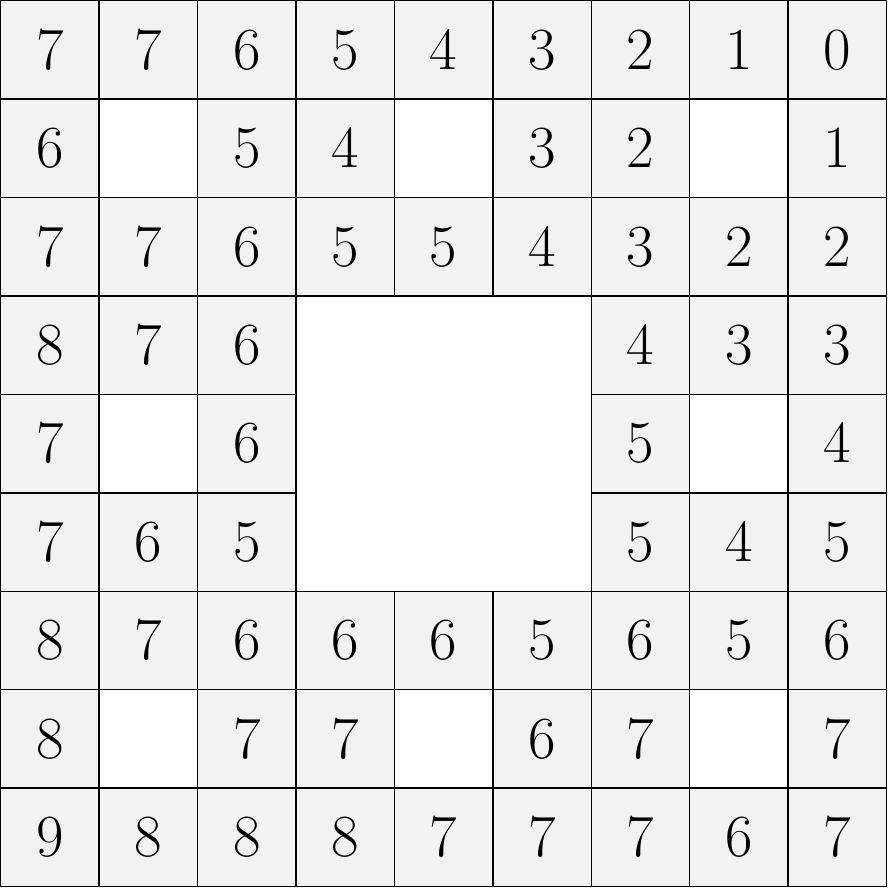}
	}
	\caption{A number in a cell indicates, among all paths entirely within $V_2$, the length of a shortest path to the top right cell. On the left, we use torus identifications, and on the right, Klein horizontal.}
	\label{fig:distances-in-V_2}
\end{figure}

Assume that, from a corner in $V_m$, any cell in $V_m$ can be reached within $r_m$ steps. Consider a corner cell $x$ of $V_{m+1}$. Starting at $x$, within $r_m$ steps we can reach a cell at the corner of the $m$-stitch of $V_{m+1}$. From here, we can reach a corner of each of the other seven copies of $V_m$ within six steps. (These six steps are the steps through some of the red cells in Figure~\ref{fig: corner-to-corner paths}, or through their mirror image after a diagonal reflection.) Applying the inductive hypothesis again, we can reach any cell in each of these other copies of $V_m$ within an additional $r_m$ steps. Adding these up, we can go from $x$ to any other cell in $V_{m+1}$ within $2 r_m + 6 = r_{m+1}$ steps.

For Klein bottle identifications, each cell of $V_2$ can be reached from the corner cell within nine steps. Again this can be verified by hand, as in Figure~\ref{fig:distances-in-V_2}. Finally in $V_{m+1}$, any two copies of $V_m$ can be joined by a path across the $m$-stitch. Such a path can be found with at most six steps, like the red path in the torus example of Figure~\ref{fig: corner-to-corner paths}.
\end{proof}

\begin{lem}
	\label{lem:lower-bound-torus-klein}
	For every cell $x$, we have $\# B(x, 2 r_m + 1) \ge 8^m$.
\end{lem}
\begin{proof}
By Lemma~\ref{lem: can reach entire V_m in r_m steps}, we can travel from $x$ to a corner of the copy of $V_m$ containing $x$ within $r_m$ steps, and from there we can travel to any other cell in the same copy of $V_m$ with an additional $r_m$ steps. Hence $B(x, 2 r_m + 1) \supseteq V_m$, and so
\[ \# B(x, 2 r_m + 1) \ge 8^m. \qedhere \]
\end{proof}


\subsection{A Lower Bound for Projective Identifications}
\label{sec:cardinality-projective-lower-bound}

Let us now consider the case of projective identifications. We obtain a looser bound. The issue is that with projective identifications, an $m$-stitch does not quickly connect all copies of $V_m$, and so we must use longer paths.

Consider a sequence $R_m$ satisfying
\begin{align}
	\label{eq:proj-lower:R_m-defn}
	R_0 = 0
	\qand
	R_{m+1} = \max \Big\{
		2 R_m + 3 , \quad
		2 R_m + 2^m + 1
	\Big\}.
\end{align}
\begin{lem}
	\label{lem:proj-lower:r_m-steps-throughout-V_m}
	Given $R_m$ steps, a path from any cell of $V_m$ can reach any other cell in $V_m$.
\end{lem}
\begin{proof}
It suffices to prove the induction step, as the base case $m = 0$ is trivial. Suppose the path goes from cell $x$ to cell $y$. First suppose $x$ is in the top copy of $V_m$, as highlighted in Figure~\ref{fig:proj-lower:abcde}(a). Inductively from $x$, it takes at most $R_m$ steps to reach a corner cell by the $m$-stich, \textit{e.g.}, a blue or green cell in Figure~\ref{fig:proj-lower:abcde}(b). From there, it can reach any other copy of $V_m$ within three steps, and then it takes at most another $R_m$ steps to reach $y$.

\begin{figure}[h]
	\centering
	\makebox[\linewidth]{
	\begin{minipage}{2cm}
		\centering
		\includegraphics{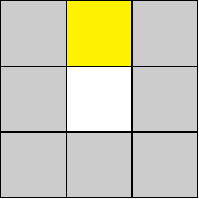} \\*
		(a)
	\end{minipage}
	\hspace{0.5cm}
	\begin{minipage}{2cm}
		\centering
		\includegraphics{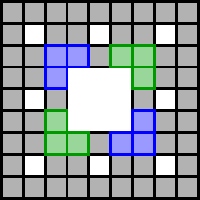} \\*
		(b)
	\end{minipage}
	\hspace{0.5cm}
	\begin{minipage}{2cm}
		\centering
		\includegraphics{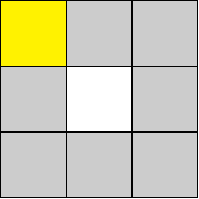} \\*
		(c)
	\end{minipage}
	\hspace{0.5cm}
	\begin{minipage}{2cm}
		\centering
		\includegraphics{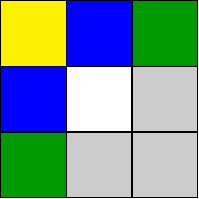} \\*
		(d)
	\end{minipage}
	\hspace{0.5cm}
	\begin{minipage}{2cm}
		\centering
		\includegraphics{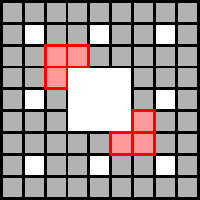} \\*
		(e)
	\end{minipage}
	}
	\caption{}
	\label{fig:proj-lower:abcde}
\end{figure}

Next suppose $x$ is in the corner copy, as in Figure~\ref{fig:proj-lower:abcde}(c). If $y$ is not in one of the green copies of $V_m$ in Figure~\ref{fig:proj-lower:abcde}(d), a path from $x$ can reach one of the red cells of Figure~\ref{fig:proj-lower:abcde}(e) using $R_m$ or fewer steps. Thus $x$ and $y$ are at most $2 R_m + 3$ steps apart.

If $y$ is in a green copy, the path can reach $y$ by traversing one of the blue copies shown in Figure~\ref{fig:proj-lower:abcde}(d); this can be achieved in $2^m - 1$ steps straight across the blue $V_m$, as in Figure~\ref{fig:proj-lower:cross-V_m-basic-line}.
\begin{figure}
	\centering
	\includegraphics{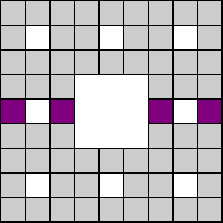}
	\caption{A path across $V_m$ of length $2^m - 1$.}
	\label{fig:proj-lower:cross-V_m-basic-line}
\end{figure}
In addition to the (upper bound of) $2 r_m$ steps needed to get from $x$ to the outer boundary of a blue copy and to get from the outer boundary of a blue copy to $y$, it takes one step to enter the blue copy and one step to leave the blue copy.
Hence the path has at most $2 R_m + 2^m - 1 + 2$ steps, and the result follows.
\end{proof}

\begin{lem}
	When $m \geq m_0 \geq 2$,
	\[ R_m \le \frac{1}{2} \left( m \cdot 2^m \right) + \left( \frac{R_{m_0} + 1}{2^{m_0}} - \frac{m_0}{2} \right) \left( 2^m \right) - 1 . \]
\end{lem}
\begin{proof}
Note that when $m \geq 2$,
\[ R_m = 2 R_{m-1} + 2^{m-1} + 1. \]
Therefore, when $k = 1$,
\begin{align*}
	R_m \le 2^k R_{m-k} + k \cdot 2^{m-1} + 2^k - 1
\end{align*}
We prove that this inequality holds for $1 \le k \le m - m_0$ by induction on $k$ as follows:
\begin{align*}
	R_m
	&\le 2^k R_{m-k} + k \cdot 2^{m-1} + 2^k - 1 \\
	&\le 2^k \Big( 2 R_{m-k-1} + 2^{m-k-1} + 1 \Big) + k \cdot 2^{m-1} + 2^k - 1 \\
	&\le 2^{k+1} R_{m-(k+1)} + (k+1) \cdot 2^{m-1} + 2^{k+1} + 1.
\end{align*}

Then, for $k = m - m_0$,
\begin{align*}
	R_m
	&\le 2^{m - m_0} R_{m_0} + (m - m_0) \cdot 2^{m - 1} + 2^{m - m_0} - 1 \\
	&= \frac{1}{2} \left( m \cdot 2^m \right) + \left( \frac{R_{m_0} + 1}{2^{m_0}} - \frac{m_0}{2} \right) \left( 2^m \right) - 1. \qedhere
\end{align*}
\end{proof}

\begin{lem}
	\label{lem:proj-lower-8^m}
	For every cell $x$ and $m$ large enough, $\# B(x, m \cdot 2^m) \ge 8^m$.
\end{lem}
\begin{proof}
We have $\#B(x, R_m) \ge 8^m$ by Lemma~\ref{lem:proj-lower:r_m-steps-throughout-V_m}, and $R_m \le m \cdot 2^m$ for large enough $m$.
\end{proof}


\subsection{The Upper Bound}
We present the upper bound in the case of any of our three identifications: torus, Klein, or projective. The main idea is if a path has fewer than $2^m - 1$ steps, then it cannot cross a copy of $V_m$(Lemma~\ref{lem:cannot-cross-V_m}), hence is trapped in certain neighboring copies of $V_m$ (Corollary~\ref{cor: must share a vertex}). Crossing is made rigorous using $m$-edges, which we now define.

If $\ell$ is a line segment in $IMC$ along which two distinct copies of $V_m$ intersect after identifications, then call $\ell$ an \defnterm{$m$-edge}. Denote by $E_{V, W}$ the $m$-edge along which the copies $V$ and $W$ of $V_m$ intersect. Say that a cell is on an $m$-edge if one of its edges lies on the $m$-edge.

An $m$-edge is \defnterm{vertical} (or \defnterm{horizontal}, respectively) if it is vertical (horizontal) in the IMC before identification (more precisely, if its preimage before identification is a union of vertical edges). Two copies of $V_m$ are \defnterm{horizontal neighbors} (or \defnterm{vertical neighbors}, respectively) if they intersect along a vertical (horizontal) $m$-edge.

\begin{lem}
	\label{lem:cannot-cross-V_m-quickly}
	Within a copy of $V_m$, there are $2^m$ columns of cells that do not hit any $m'$-stitch for $m' \le m-1$.
\end{lem}
\begin{proof}
Starting with $m = 0$, for which there is indeed $1 = 2^0$ column, we induct on $m$. Consider $V_{m+1}$, which contains eight copies of $V_m$. Notice that the three copies of $V_m$ on the left of the central $m$-stitch---and hence the $2^m$ columns of cells they contain---stack one on top of the next (Figure~\ref{fig:2^mcolumns-in-V-m}). Since there is an $m$-stitch, the same does not happen in the center, but it does happen on the right side; so, there are $2 \cdot 2^m = 2^{m+1}$ columns in $V_{m+1}$ that do not hit any $m$-stitch, as desired.
\begin{figure}[h]
	\centering
	\begin{minipage}{4cm}
		\centering
		\includegraphics[width=0.75cm]{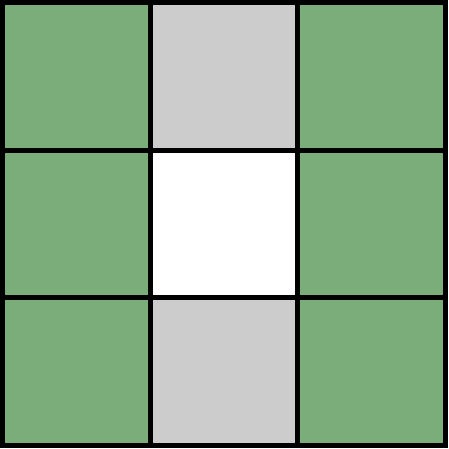}
	\end{minipage}
	\hspace{1cm}
	\begin{minipage}{4cm}
		\centering
		\includegraphics[width=2.25cm]{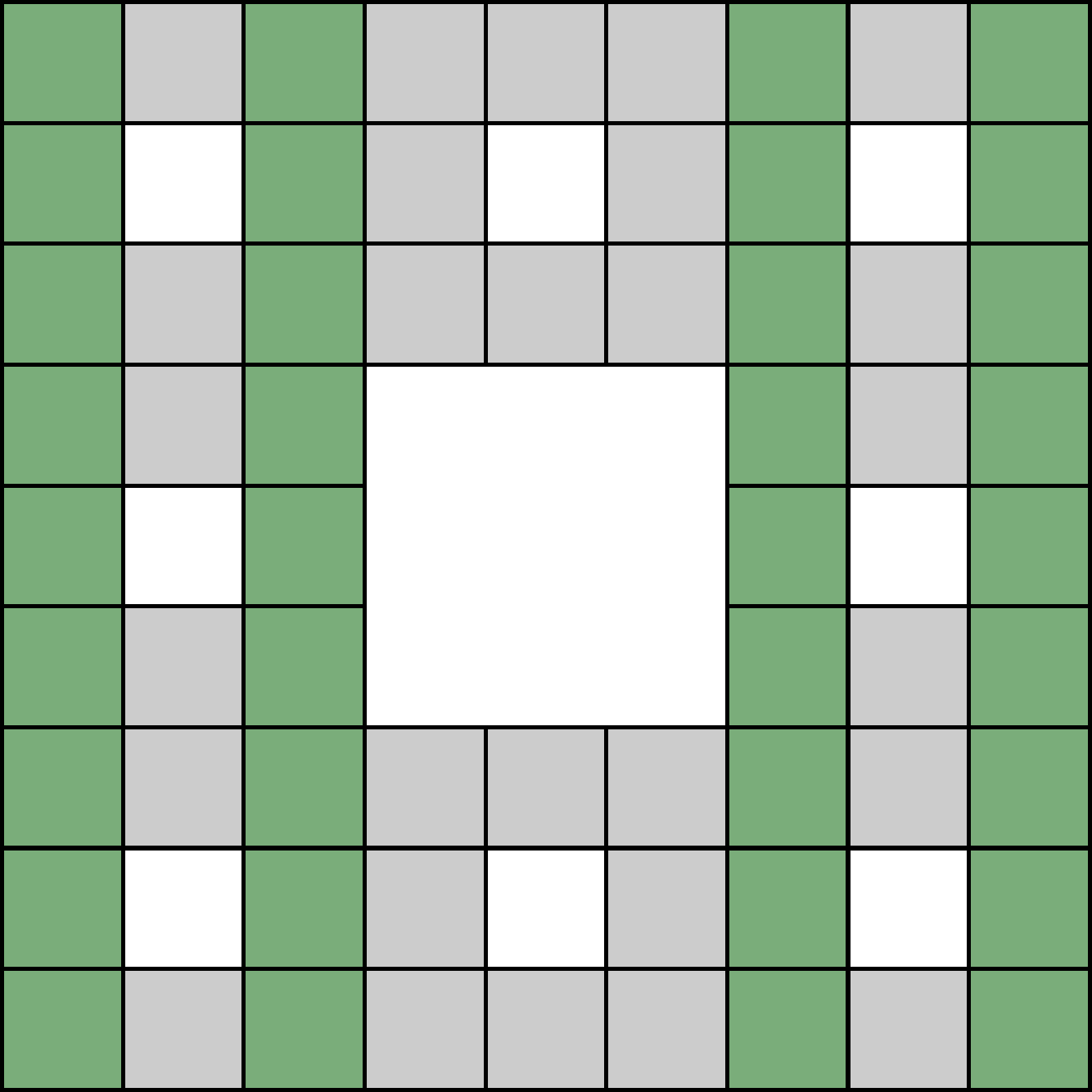}
	\end{minipage}
	\caption{The two columns in $V_1$ and the four columns in $V_2$}
	\label{fig:2^mcolumns-in-V-m}
\end{figure}
\end{proof}

\begin{lem}
	\label{lem:cannot-cross-V_m}
	Consider a path inside a copy of $V_m$ with at most $2^m - 1$ steps. If the path begins along an $m$-edge of the copy of $V_m$, then it cannot leave along the opposite $m$-edge.
\end{lem}
\begin{proof}
Without loss of generality, assume a path goes between the left and right sides. From Lemma~\ref{lem:cannot-cross-V_m-quickly} we obtain $2^m$ columns of cells in $V_m$ that do not intersect any $m'$-stitches for $m' \le m-1$. Since the path is constrained to $V_m$, it cannot use any $M$-stitches for $M \ge m$, so it must traverse each of these $2^m$ columns. This requires $2^m - 1$ steps, and so the path cannot leave the copy of $V_m$ through the opposite $m$-edge.
\end{proof}

Lemma~\ref{lem:cannot-cross-V_m} prevents paths that are too short from crossing a copy of~$V_m$.

\begin{ex}
In Figure~\ref{fig: cross V_m basic}, a path cannot connect the blue $m$-edge to the green $m$-edge without leaving the copy of $V_m$ shown, unless it has $2^m - 1$ or more steps.
\begin{figure}[h]
	\centering
	\includegraphics{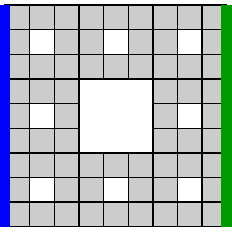}
	\caption{}
	\label{fig: cross V_m basic}
\end{figure}
\end{ex}

Our last result was restricted to paths inside a particular copy of $V_m$; we must now remove this restriction. Our goal is to show first that a short path remains within a few consecutive copies of $V_m$, and second that these copies all share a vertex. Because our path must whirl around this common vertex, we shall call it the center. To make this precise, let us first introduce $m$-stacks, $m$-sequences, and $m$-segments-of-two.

An \defnterm{$m$-stack} is a finite sequence of copies of $V_m$ such that consecutive copies are all horizontal neighbors or all vertical neighbors. A cell is in an $m$-stack if it is in a copy of $V_m$ in the stack. Observe that every $m$-stack is isomorphic as a cell graph to a sequence of copies of $V_m$ with the bottom edge of each copy glued to the top edge to the previous one with edge orientations preserved (Figure~\ref{fig:m-stack-examples}). Hence, every $m$-stack has a rectangular outer boundary. A \defnterm{side} of an $m$-stack is a side of this outer boundary. A cell is on a side if it intersects with the side.

\begin{figure}
	\centering
	\includegraphics[width=7cm]{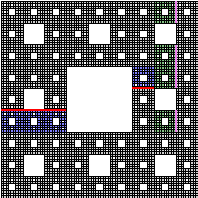}
	\hspace{1.5cm}
	\includegraphics[height=7cm]{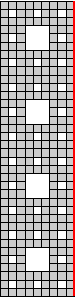}
	\caption{Two different $2$-stacks (one blue, one green, and containing four copies of $V_2$ each) in a copy of $V_4$ with Klein identifications. Both $2$-stacks are equivalent to the rectangular $2$-stack shown on the right.}
	\label{fig:m-stack-examples}
\end{figure}

Lemma~\ref{lem:cannot-cross-V_m} essentially provides the following result, as well:

\begin{cor}
	\label{lem:path-of-length-2^m-1}
	If a path is contained in an $m$-stack and has cells on two opposite sides, then its length is at least $2^m - 1$.
\end{cor}

Consider a finite path $\gamma$ as a sequence of cells: $c_0, c_1, c_2, \dotsc, c_n$. From $\gamma$ we form a sequence of copies of $V_m$ as follows:
\begin{enumerate}
	\item For each $c_i$, find the copy $V_m^{c_i}$ of $V_m$ containing $c_i$, and form a new sequence $V_m^{c_1}, \dotsc, V_m^{c_n}$.
	\item While consecutive elements in this new sequence are equal, delete all but one of them.
\end{enumerate}
What remains after these deletions we call the \defnterm{$m$-sequence} of $\gamma$. In other words, the first copy in the $m$-sequence is the copy of $V_m$ that contains the starting cell, the second term is the copy containing the first cell $c_i$ that is not in the first copy, the third copy is that containing first cell after $c_i$ that is not in the second copy, and so on.

For an $m$-sequence, we define an \defnterm{$m$-segment-of-two} as a tuple $(i, V, W)$ such that $V$ is the $i$th copy of $V_m$ in the $m$-sequence, $W$ is the $(i + 1)$st copy, and neither $V$ nor $W$ is the $(i-1)$st copy (if it exists). The $m$-segments-of-two from a particular $m$-sequence are clearly ordered by their first entries. Note that if $(i, V, W)$ is an $m$-segment-of-two, then $V$ and $W$ form an $m$-stack.

The $j$th copy of $V_m$ in the $m$-sequence of a path is in an $m$-segment-of-two $(i, V, W)$ if and only if all copies inclusively between the $i$th and the $j$th copies are either $V$ or $W$. A cell $c$ is in an $m$-segment-of-two if and only if the copy of $V_m$ containing $c$ is in the $m$-segment-of-two. Observe that all copies of $V_m$ in an $m$-segment-of-two $(i, V, W)$ are either $V$ or $W$.

\begin{ex}
\label{example:m-sequence-of-a-path}
Fix $m = 1$ and consider the blue path through $V_2$ shown  in Figure~\ref{fig:example-m-sequence}. The figure marks four copies of $V_1$, namely $U$, $V$, $W$ and $X$. The $m$-sequence of the path is $U, V, U, W, X$. There are three $m$-segments-of-two: the first, $(1, U, V)$, contains the first three terms of the $m$-sequence; the next, $(3, U, W)$, contains the third and the fourth terms; and the last, $(4, W, X)$, contains the final two terms.
\begin{figure}[h]
	\centering
	\includegraphics[width=6cm]{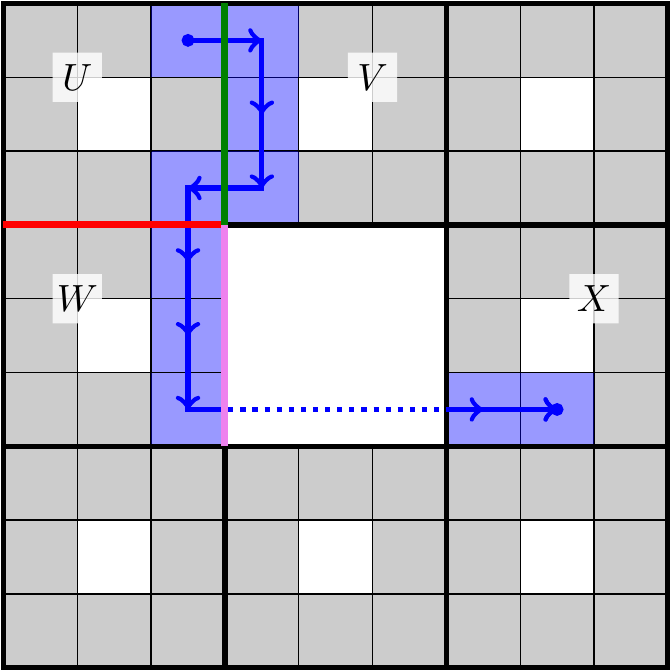}
	\caption{A sample path in $V_2$ (proceeding downwards, roughly, and then to the right).}
	\label{fig:example-m-sequence}
\end{figure}
\end{ex}

\newcommand{\Enter}[1]{\textrm{Ent}(#1)}
\newcommand{\Exit}[1]{\textrm{Exit}(#1)}

A path $(c_0, c_1, \dotsc)$ is said to \defnterm{enter} an $m$-segment-of-two $(i, V, W)$ through an $m$-edge if there are consecutive cells $c_n, c_{n+1}$ so that (1)~both are on the $m$-edge, and (2)~only $c_{n+1}$ is within $(i, V, W)$. The notion of a path \defnterm{exiting} is analogous, with $c_n$ in $(i, V, W)$ instead. The $m$-edge along which the path enters $(i, V, W)$ is denoted by $\Enter i$, and the $m$-edge along which it exits is denoted $\Exit i$.

\begin{ex}[Continuing Example~\ref{example:m-sequence-of-a-path}]
Figure~\ref{fig:example-m-sequence} distinguishes the following $m$-edges:
\begin{center}
	\begin{tabular}{|c|cc|cc|} \hline
		$m$-segment-of-two & Enter & Color & Exit & Color \\ \hline
		$(1, U, V)$ & {-} & {-} & $\Exit 1$ & red \\
		$(3, U, W)$ & $\Enter 3$ & green & $\Exit 3$ & pink \\
		$(4, W, X)$ & $\Enter 4$ & red & {-} & {-} \\ \hline
	\end{tabular}
\end{center}
\end{ex}

Recall here that if $V$ and $W$ are copies of $V_m$, then $E_{V, W}$ is the $m$-edge along which $V$ and $W$ intersect.

\begin{lem}
	Let $(i, V, W)$ be an $m$-segment-of-two associated to a path of length at most $2^m$.
	\begin{enumerate}
		\item Each of $\Enter i$ and $\Exit i$ (assuming it exists) is distinct from $E_{V, W}$, although it intersects $E_{V, W}$.
		\item If $\Enter i$ and $\Exit i$ both exist (\textit{i.e.,} the path comes from and goes to other copies of $V_m$), then they are on the same side of the $m$-stack formed by $V$ and $W$.
	\end{enumerate}
\end{lem}
\begin{proof}
For the first claim, it suffices to consider $\Enter i$; the case for $\Exit i$ is similar. By the definitions of $m$-segments-of-two and entering, $\Enter i$, if it exists, cannot be $E_{V, W}$. Since the path must go from $V$ to $W$, it has cells on $E_{V, W}$; furthermore, since the path has at most $2^m$ steps, and one step is required to enter the $m$-segment-of-two, Lemma~\ref{lem:path-of-length-2^m-1} prevents it from entering through an $m$-edge parallel to $E_{V, W}$. The first claim follows, for the remaining edges of the $m$-stack formed by $V$ and $W$ all intersect $E_{V, W}$.

For the second statement, suppose $\Enter i$ and $\Exit i$ are on different sides of the $m$-segment-of-two. Counting the steps needed to enter and exit, Lemma~\ref{lem:path-of-length-2^m-1} implies the path has more than $2^m$ cells, a contradiction.
\end{proof}

Consider a path with at most $2^m$ steps and at least two $m$-segments-of-two. Since the path has at least two $m$-segments-of-two, for each $m$-segment-of-two $(i, V, W)$, at least one of $\Enter i$ and $\Exit i$ exists. Define the \defnterm{center} of the $m$-segment-of-two $(i, V, W)$ to be the vertex where $E_{V, W}$, $\Enter i$ and $\Exit i$ (or those that exist) intersect. Notice that the center will always be a corner vertex of $V$ and $W$. The lemma above ensures the center is well defined.

\begin{ex}
Figure~\ref{fig:center-example} shows a green path of length~3 and a blue path of length~4. Forming the 0-sequence associated to the green path, we see its center is marked by the pink dot. Due to the torus identifications taken in this picture, the center of the blue path is represented by four points, the red dots.
\begin{figure}[h]
	\centering
	\includegraphics[width=8cm]{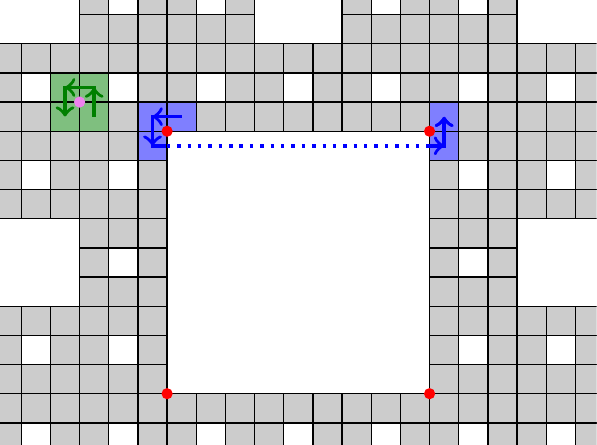}
	\caption{The centers of two paths.}
	\label{fig:center-example}
\end{figure}
\end{ex}

\begin{lem}
	\label{lem:m-segments-of-two-have-same-centers}
	Consider a path with at most $2^m$ steps and at least two $m$-segments-of-two. The centers of the $m$-segments-of-two are the same.
\end{lem}

\begin{proof}
Let us limit our focus to consecutive $m$-segments-of-two: $(i, V, W)$ and $(j, X, Y)$, where $i < j$, $X \in \{V, W \}$, and $Y \notin \{ V, W \}$. We show that the centers are the same, and induction completes the argument. Note that the path enters $(j, X, Y)$ through $E_{V, W}$ and exits $(i, V, W)$ at $E_{X, Y}$. The center of $(i, V, W)$, therefore, is where $\Exit i = E_{X, Y}$ intersects $E_{V, W}$, and the center of $(j, X, Y)$ is where $\Enter j = E_{V, W}$ intersects $E_{X, Y}$. Hence, the intersection of $E_{V, W}$ and $E_{X, Y}$ marks both centers.
\end{proof}

\begin{lem}
	For a path of length at most $2^m$, let $U$ be the first copy of $V_m$ in the associated $m$-sequence. There exists a corner vertex $x$ of $U$ shared by every copy of $V_m$ in the $m$-sequence.
\end{lem}
\begin{proof}
First, if $U$ is the only copy of $V_m$ in the $m$-sequence, the result holds. Second, when there is only one $m$-segment-of-two, $(1, V, W)$, we may choose any vertex shared by the two relevant copies $V$ and $W$.

If the path has at least two $m$-segments-of-two, then the centers of all $m$-segments-of-two are the same by Lemma~\ref{lem:m-segments-of-two-have-same-centers}. In particular, the centers are the same as that of the first $m$-segment-of-two. Since the center of the first $m$-segment-of-two is a vertex of $U$, the result follows from the fact that every $m$-cell touches the center of an $m$-segment-of-two.
\end{proof}

The proof shows the center often suffices for this common vertex. The only time it does not is for paths with one $m$-segment-of-two, when the center is undefined. The existence of this vertex then gives:

\begin{cor}
	\label{cor: must share a vertex}
	Consider $z \in IMC$, and denote by $V_z$ the copy of $V_m$ containing $z$. Define $\cV$ to be the set of all copies of $V_m \subset IMC$ that share a corner vertex with $V_z$. Any path from $z$ with length at most $2^m - 1$ remains within $\bigcup_{V \in \cV} V$; that is,
	\[ B\left( z, 2^m \right) \subseteq \bigcup_{V \in \cV} V. \]
\end{cor}

\begin{lem}
	\label{lem: counting neighbors}
	With the notation of Corollary~\ref{cor: must share a vertex}, $\# \cV \le 45$.
\end{lem}
\begin{proof}
Consider a corner $x$ of $V_z$. If $x$ is a corner of a stitch, then $V_z$ touches at most 11~other copies of $V_m$ at $x$ (Figure~\ref{fig:cardinality-share-vertex-through-stitch-with-eleven-copies}). Otherwise, it touches only 3~other copies of $V_m$ at $x$. Counting at most 11~copies for each of the four outer vertices of $V_z$, along with $V_z$ itself, we have
\[ \# \cV \le 4 \cdot 11 + 1 = 45. \qedhere \]
\begin{figure}[h]
	\centering
	\includegraphics{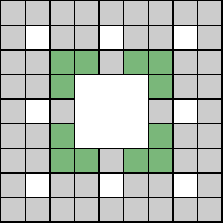}
	\caption{A copy of $V_m$ can share a vertex with eleven other copies through the stitch.}
	\label{fig:cardinality-share-vertex-through-stitch-with-eleven-copies}
\end{figure}
\end{proof}

\begin{remark}
	The count of~45 in Lemma~\ref{lem: counting neighbors} is sufficient for our purposes, but it can be improved to~21. Doing so improves the constant~2880 appearing within the proof of Theorem~\ref{thm:full-bound-torus-klein}.
\end{remark}


\subsection{Full Bound}
Here we combine the upper and lower bounds above, first for torus and Klein identifications and then for projective identifications.

\begin{thm}
	\label{thm:full-bound-torus-klein}
	Let $x \in IMC$ with either torus or Klein identifications. Then $\# B(x, r) \sim r^3$; that is, there are constants $c$ and $C$ so that, when $r$ is sufficiently large,
	\[ c r^3 \le \# B(x, r) \le C r^3. \]
\end{thm}
\begin{proof}
For any $r$, we can find $m$ so that
\[ 2^{m-1} \le r + 1 \le 2^m. \]
With the notation of Section~\ref{sec:cardinality-torus-klein}, specifically~\eqref{eq: r_m closed form}, we can write
\[ 2 r_{m-4} + 1
	 = 2^{m-1} - 11
	 \le r \le 2^m. \]
Lemma~\ref{lem:lower-bound-torus-klein} then gives us the lower bound
\[ \# B(x, r) \ge \# B(x, 2r_{m-4} + 1) \ge 8^{m-4}. \]
Further, Corollary~\ref{cor: must share a vertex} and Lemma~\ref{lem: counting neighbors} yield the upper bound
\[ \# B \left(x, r \right) \le \# B \left(x, 2^m \right) \le \# \cV \cdot 8^m \le 45 \cdot 8^m. \]
Additionally, if $r$ is sufficiently large, $2^{m-2} \le r \le 2^m$, in which case
\[ \frac{1}{8^m} \le \frac{1}{r^3} \le \frac{1}{8^{m-2}}. \]
Apply the lower and upper bounds to obtain
\[ \frac{8^{m-4}}{8^m} \le \frac{\# B(x, 2 r_{m-4} + 1)}{r^3} \le \frac{\# B(x, r)}{r^3} \le \frac{\# B(x, 2^m)}{r^3} \le \frac{45 \cdot 8^m}{8^{m-2}}. \]
Hence
\[ \frac{1}{4096} r^3 \le \# B(x, r) \le 2880 r^3. \qedhere \]
\end{proof}

\begin{remark}\label{rmk: ball-cardinalty-with-mixed-identifications}
The upper bound holds on the infinite magic carpet with any identifications. The lower bound holds even when torus and Klein styles are mixed together, so long as there are no projective identifications. Hence $\#B(x,r) \sim r^3$ for any mixture of torus and Klein identifications.
\end{remark}

Now we return to projective identifications.

\begin{thm}
	With projective identifications, for large enough $r$,
	\[ c \left( \frac{r}{\log r} \right)^3 \leq \# B(x, r) \leq C r^3. \]
\end{thm}
\begin{proof}
The upper bound holds as mentioned above, so it suffices to consider the lower bound.
Choose $m \cdot 2^m < r \le (m+1) \cdot 2^{m+1}$; then
\[ \log m + m \log 2 < \log r, \]
in which case $m \log 2 < \log r$. Then, for sufficiently large $r$,
\[ m + 1 < 2 \cdot \frac{\log r}{\log 2}. \]
Using this to substitute for $\log r$ along with Lemma~\ref{lem:proj-lower-8^m}, we find
\begin{align*}
	\frac{\# B(x, r)}{\left( r / \log r \right)^3}
	&\ge \frac{8^m \cdot (\log r)^3}{\left( (m+1) \cdot 2^{m+1} \right)^3} \\
	&\ge \frac{8^m \cdot (m+1)^3 \left(\log 2 / 2 \right)^3}{\left( (m+1) \cdot 2^{m+1} \right)^3} \\
	&= \frac{8^{m-1} \cdot (\log 2)^3}{8^{m+1}} \\
	&= \frac{(\log 2)^3}{64}.
	\qedhere
\end{align*}
\end{proof}


\subsection{The Cardinality Ratio}
Let us briefly consider the cardinality ratio $\#B(x, r) / r^3$. We show how this ratio behaves with each identification type around two sample points in Figure~\ref{fig:cardinality-ratio-plots}.
\begin{figure}
	\centering
	\newcommand{\picturewidth}{8cm}
	\makebox[\linewidth]{
		\includegraphics[width=\picturewidth]{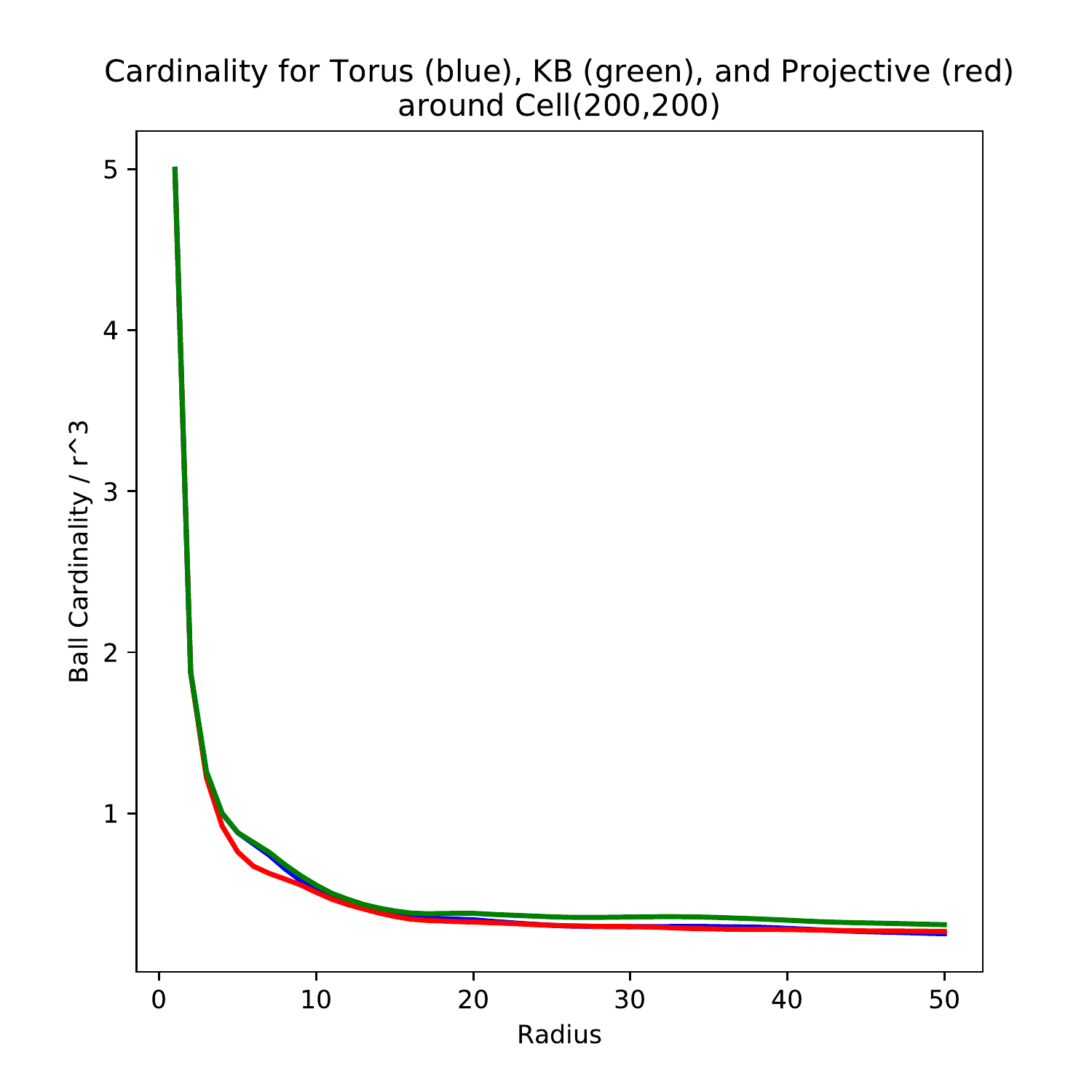}
		\includegraphics[width=\picturewidth]{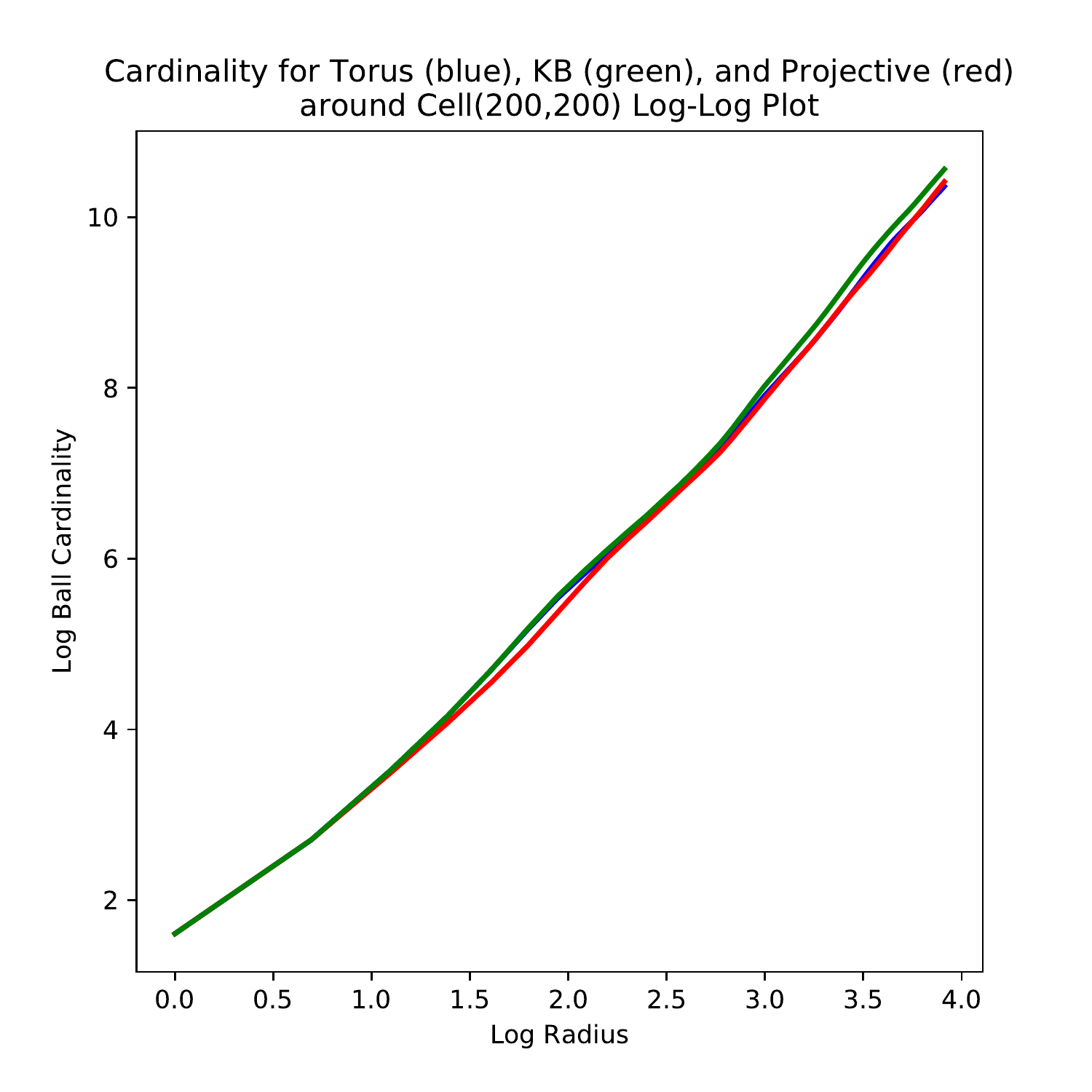}
	}
	\makebox[\linewidth]{
		\includegraphics[width=\picturewidth]{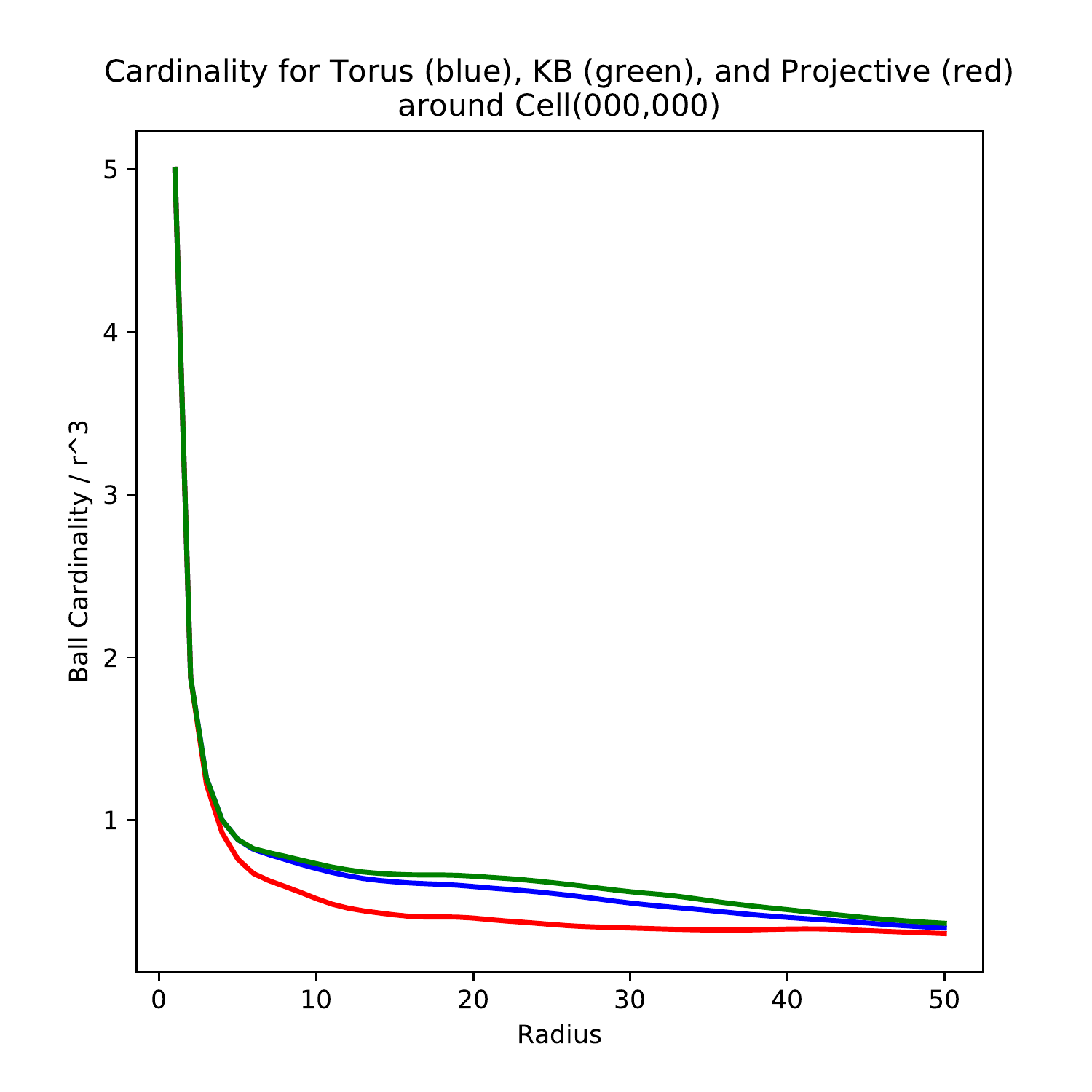}
		\includegraphics[width=\picturewidth]{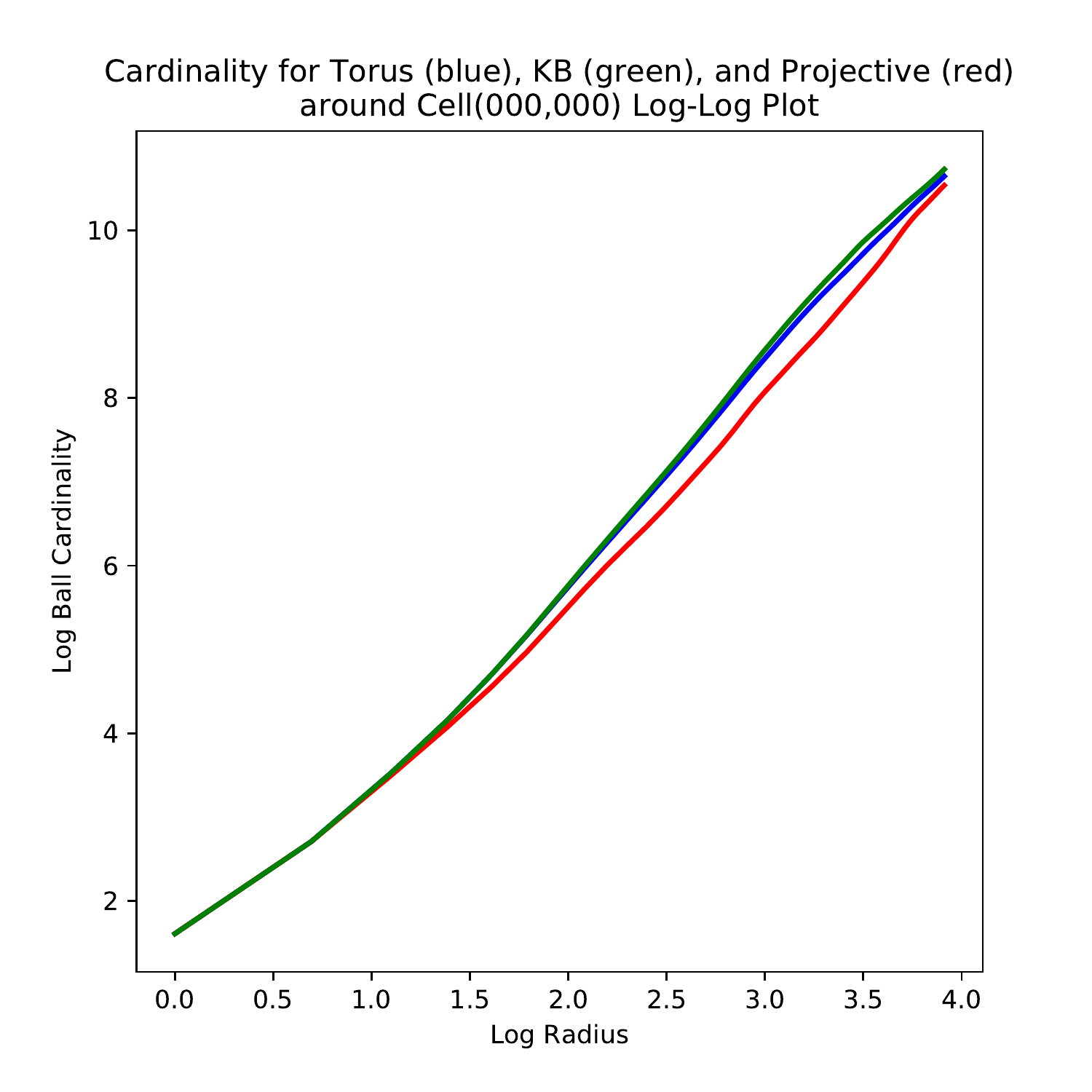}
	}
	\let \picturewidth \relax
	\caption{The ball cardinality ratio $\#B(x,r)/r^3$ around two sample points (left), each beside a corresponding $\log \#B(x,r)$-to-$\log r$ plot (right).}
	\label{fig:cardinality-ratio-plots}
\end{figure}
The plots suggest this ratio may converge, although we do not have proof of this. Of course, this conjecture is less certain for projective identifications, where even the $\sim r^3$ growth rate is unknown.  We can say, however, that if this ratio converges, then it will converge consistently over all cells.
\begin{prop}
	Fix any identification type and any two cells $x$ and $y$. If
	\[ \lim_{r \to \infty} \frac{\# B(x, r)}{r^3} = c,
			\qqt{then also}
			\lim_{r \to \infty} \frac{\# B(y, r)}{r^3} = c. \]
\end{prop}
\begin{proof}
For large $r$, the triangle inequality provides $\# B(x, r - d(x,y)) \le \# B(y, r) \le \# B(x, r + d(x,y))$, from which
\[ \lim_{r \to \infty} \frac{\# B(x, r-d(x,y))}{r^3}
		\le \lim_{r \to \infty} \frac{\# B(y, r)}{r^3}
		\le \lim_{r \to \infty} \frac{\# B(x, r+d(x,y))}{r^3}. \]
Our assumption $\#B(x,r)/r^3 \to c$ may be used to compute these bounds. For the lower bound,
\[ \lim_{r \to \infty} \frac{\# B(x, r-d(x,y))}{r^3} = \lim_{r \to \infty} \frac{\# B(x, r-d(x,y))}{(r - d(x,y))^3} \cdot \frac{(r - d(x,y))^3}{r^3} = c, \]
and similarly for the upper bound.
\end{proof}


\section{Random Walks}
\label{sec: random walks}

In this section, we present data of our computer simulation concerning random walk on the $IMC$ and the effective resistance from a fixed point to the boundaries of large squares.

The random walk simulation was carried out on the $IMC$ obtained by applying the repeated application of the inverses of the contractions that fix two opposite vertices. The starting points were chosen to be the cell whose lower left hand corner is $(0, 0)$ before identification. Only the simple symmetric random walk was considered, \textit{i.e.}, the random walker has equal probability, $1/4$, of moving upwards, downwards, to the left and to the right. A trial terminates either when the random walker has returned to the starting point, in which case the walk is said to be empirically recurrent, or when the walker has walked a prescribed number of steps, the maximum length. If a trial is not empirically recurrent, it is said to be empirically transient. The length of a trial is the number of steps the walker has walked when the trial terminates.

We note that in each simulation, roughly $2/3$ of all trials are empirically recurrent. Even though the computations are not conclusive, they suggest the walk is transient. The results of the simulations are summarized in Table~\ref{table:random-walk-simulations}.
\begin{table}[b]
	\centering
	\begin{tabular}{| r r r c |}
		\hline
		identification & no.\ of trials &	max.\ length & \parbox{4.5cm}{\centering number (percentage) of empirically recurrent trials} \\ \hline
		torus & 2\,000 & 500\,000 & 1\,348 (67.4\%) \\
		torus & 500 & 1\,000\,000 & 331 (66.1\%) \\
		$\dagger$ \hfill torus & 2\,000 & 10\,000\,000 & 1\,390 (69.5\%) \\
		$\ddagger$ \hfill torus & 200 & 100\,000\,000 & 137 (68.5\%) \\
		Klein bottle & 1\,000 & 500\,000 & 667 (66.7\%) \\
		projective plane & 1\,000 & 500\,000 & 683 (68.3\%) \\
		\hline
	\end{tabular}
	\tablespacing
	\caption{Results of random walk simulations on infinite magic carpets.}
	\label{table:random-walk-simulations}
\end{table}

Concerning the lengths of empirically recurrent trials, we note that most of them are very short, but the frequency graphs in Figure~\ref{fig:recurrent-walk-length-histograms} have long tails. No power law was observed. More details can be found on the website~\cite{website}.

\begin{figure}
	\centering
	\newcommand{\picturewidth}{9cm}
	\makebox[\linewidth]{
		\includegraphics[width=\picturewidth]{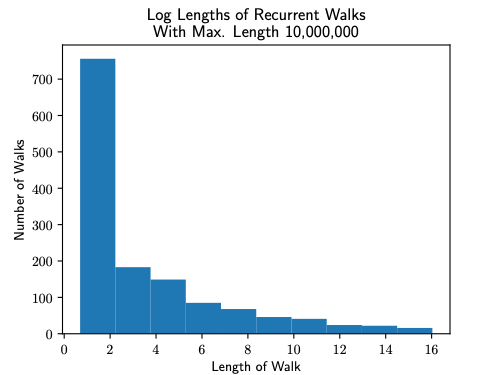}
		\includegraphics[width=\picturewidth]{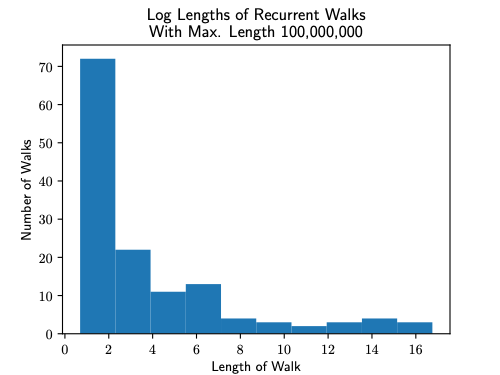}
	}
	\caption{(Left) Log lengths of random walks that returned for the simulations with maximum length set to 10,000,000, \textit{i.e.}, row $\dagger$ in Table~\ref{table:random-walk-simulations}. (Right) Similar, but for maximum length 100,000,000, \textit{i.e.}, row $\ddagger$ in Table~\ref{table:random-walk-simulations}.}
	\label{fig:recurrent-walk-length-histograms}
	\let \picturewidth \relax
\end{figure}

As for the effective resistance, the resistance from each cell to the outermost square boundary of the $m$th approximation of the $IMC$ with torus identification was computed for $m = 2, 3, 4$. High computation cost rendered direct computation impractical for larger $m$. Instead, a number of cells in the 5th approximation are randomly selected to compute their resistances. The resistance of a cell to the boundary is computed by solving the harmonic equation with the value at the cell fixed to be one and those at the boundary cells fixed to be zero. If the random walk is recurrent, the resistances should remain bounded across all $m$'s, as in the case for the random walk on $\ZZ^2$ (\textit{cf.}~\cite{doyle-snell}).

The results are summarized in Table~\ref{table:max-resistance} and Figure~\ref{fig:torus-resistance}. The resistances of squares in $\ZZ^2$ are included in Figure~\ref{fig:Z2-resistance} for comparison.

\begin{table}[h]
	\centering
	\begin{tabular}{| r c |}
		\hline
		$m$ & \parbox{3.5cm}{\centering{} max.\ resistance on $m$th approximation} \\ \hline
		2 & 0.385 \\
		3 & 0.521 \\
		4 & 0.629 \\
		\hline
	\end{tabular}
	\tablespacing
	\caption{Maximum Resistances of Cells on the $m$th Approximation of the IMC}
	\label{table:max-resistance}
\end{table}

\begin{figure}[h]
	\centering
	\newcommand{\picturewidth}{8cm}
	\makebox[\linewidth]{
		\includegraphics[width=\picturewidth]{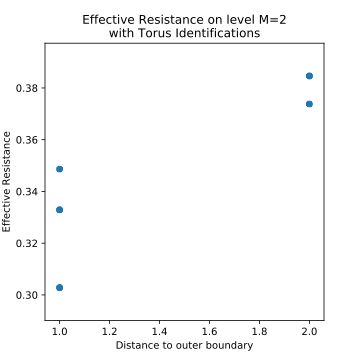}
		\includegraphics[width=\picturewidth]{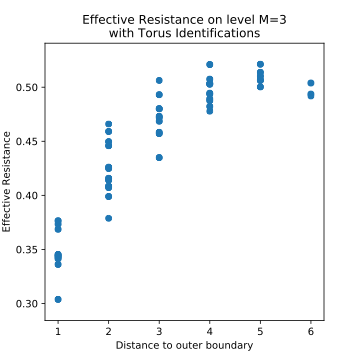}
	}
	\makebox[\linewidth]{
		\includegraphics[width=\picturewidth]{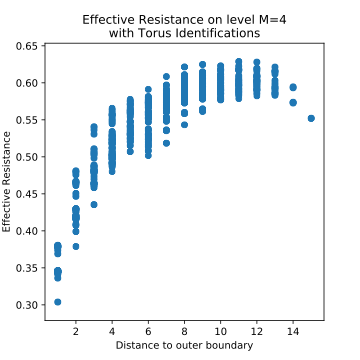}
		\includegraphics[width=\picturewidth]{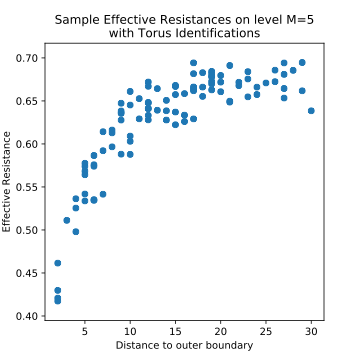}
	}
	\let \picturewidth \relax
	\caption{Effective resistances with torus identifications, at levels $M = 2$ (top left), 3 (top right), 4 (bottom left). For level 5 (bottom right), only some sample points are shown.}
	\label{fig:torus-resistance}
\end{figure}

\begin{figure}[h]
	\centering
	\newcommand{\picturewidth}{8cm}
	\makebox[\linewidth]{
		\includegraphics[width=\picturewidth]{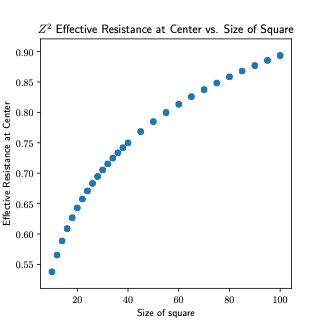}
		\includegraphics[width=\picturewidth]{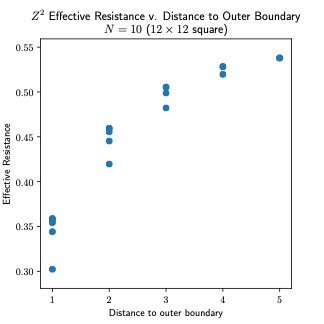}
	}
	\makebox[\linewidth]{
		\includegraphics[width=\picturewidth]{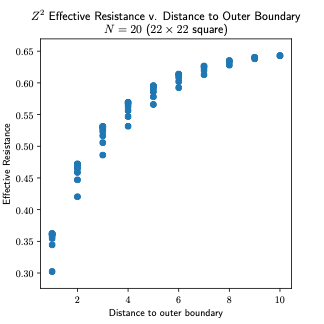}
		\includegraphics[width=\picturewidth]{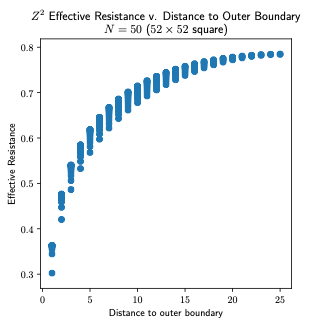}
	}
	\let \picturewidth \relax
	\caption{Effective resistances on the $\ZZ^2$ lattice}
	\label{fig:Z2-resistance}
\end{figure}

As shown in Figure~\ref{fig:torus-resistance}, the resistances follow a hill-shaped trend as the distance from the boundary varies. Unlike the case for $\ZZ^2$, in Figure~\ref{fig:Z2-resistance}, the maximum resistance for each distance does not increase monotonically as the cell moves away from the boundary, but peaks at around $2/3$ of the maximum distance. Since only data for $m = 2, 3, 4$ are gathered, it is difficult to infer the behavior of the resistances for larger $m$, and hence the nature of the random walk on the $IMC$.


\section{Spectrum of the Laplacian on \textit{IMC}}
\label{sec: spectrum on IMC}

For each of the identification types, we would like to speculate on the structure of the spectrum of the Laplacian on $IMC$ by doing calculations on $\tilde{MC}_m$ for $m \le 4$. Suppose, for example, that there were a square summable eigenfunction $\phi(x)$. Then it would have to vanish as $x \to \infty$. In particular, on $\tilde{MC}_m$ for large enough $m$ it would have to be very close to zero on a neighborhood of the boundary squares. It would also have to be close to a Dirichlet eigenfunction (one that vanishes on the boundary). So we compute all of the Dirichlet eigenfunctions and examine them to see if they are close to zero in a neighborhood of the boundary. We show some samples from the first~150 in Figure~\ref{fig:dirichlet-eigenfunctions-torus}.
\begin{figure}
	\newcommand{\picturewidth}{5cm}
	\makebox[\linewidth]{
		\includegraphics[width=\picturewidth, clip=true, trim=0.5cm 1cm 0.5cm 0cm]{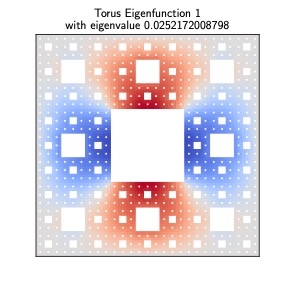}
		\includegraphics[width=\picturewidth, clip=true, trim=0.5cm 1cm 0.5cm 0cm]{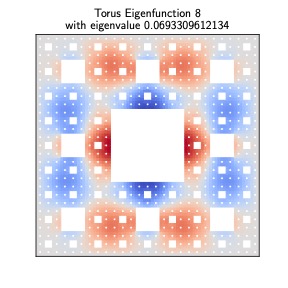}
		\includegraphics[width=\picturewidth, clip=true, trim=0.5cm 1cm 0.5cm 0cm]{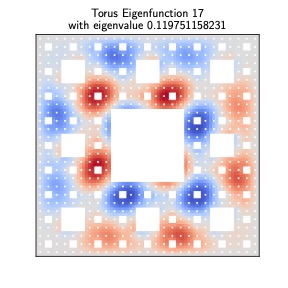}
	}
	\makebox[\linewidth]{
		\includegraphics[width=\picturewidth, clip=true, trim=0.5cm 1cm 0.5cm 0cm]{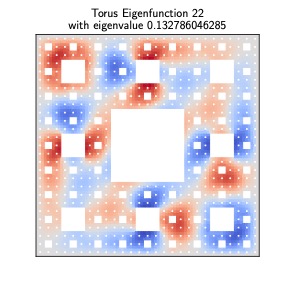}
		\includegraphics[width=\picturewidth, clip=true, trim=0.5cm 1cm 0.5cm 0cm]{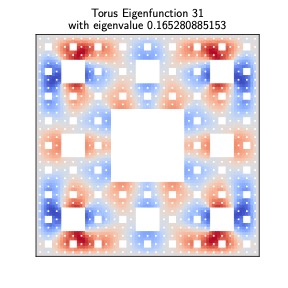}
		\includegraphics[width=\picturewidth, clip=true, trim=0.5cm 1cm 0.5cm 0cm]{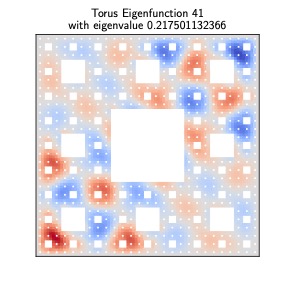}
	}
	\makebox[\linewidth]{
		\includegraphics[width=\picturewidth, clip=true, trim=0.5cm 1cm 0.5cm 0cm]{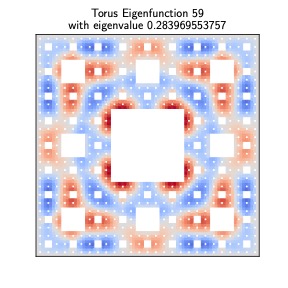}
		\includegraphics[width=\picturewidth, clip=true, trim=0.5cm 1cm 0.5cm 0cm]{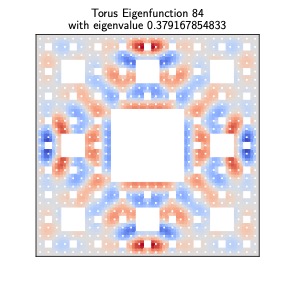}
		\includegraphics[width=\picturewidth, clip=true, trim=0.5cm 1cm 0.5cm 0cm]{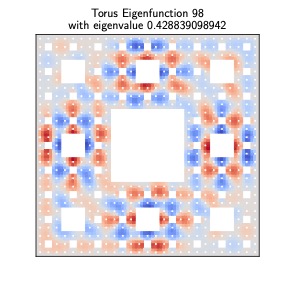}
	}
	\makebox[\linewidth]{
		\includegraphics[width=\picturewidth, clip=true, trim=0.5cm 1cm 0.5cm 0cm]{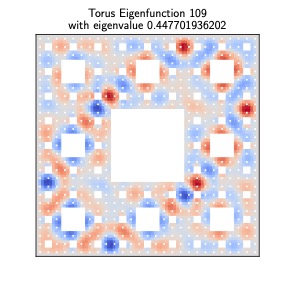}
		\includegraphics[width=\picturewidth, clip=true, trim=0.5cm 1cm 0.5cm 0cm]{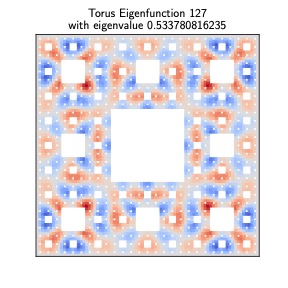}
		\includegraphics[width=\picturewidth, clip=true, trim=0.5cm 1cm 0.5cm 0cm]{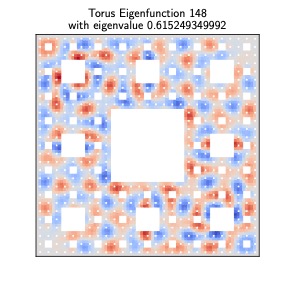}
	}
	\let \picturewidth \relax
	\caption{Some Dirichlet eigenfunctions with torus identifications.}
	\label{fig:dirichlet-eigenfunctions-torus}
\end{figure}
Many more (the first~150 for each identification type) can be found on the website~\cite{website}. None of the first~150 appears to have this decay property. We take this as evidence that the spectrum of the Laplacian on $IMC$ is entirely continuous. Of course, we were limited by our computational resources to $m \le 4$, so it is conceivable, although unlikely, that a discrete spectrum only makes an appearance at larger values of $m$. In Section~\ref{sec: uniform} we will construct a countable family of bounded periodic eigenfunctions on $IMC$. In Section~\ref{sec: spectral resolution} we conjecture that these provide the spectral resolution of the Laplacian on $IMC$ with a purely continuous spectrum.

Our Dirichlet eigenfunctions and corresponding tables of values have no relationship to the $IMC$ spectrum, but we will use them in Section~\ref{sec: heat kernel} to approximate the heat kernel on $IMC$.


\section{The Heat Kernel on \textit{IMC}}
\label{sec: heat kernel}

As mentioned in the introduction, we have computed the heat kernel on $MC_m$ for $m = 4$. For points $x, y$ far from the boundary and moderate $t$, we expect this to be a good approximation to the heat kernel on $IMC$. It is of course interesting to understand the behavior of the heat kernel on $IMC$ for large values of $t$, but we are limited by our computational resources to get a handle on this question. Complete data is available on the website~\cite{website}.

The first question we consider is the on-diagonal behavior, $H_t(x,x)$. From other fractal models we were led to expect a power law behavior \cite{Barlow98}, but instead found that power varies with the point. In Figure~\ref{fig:heat-kernel-diagonal-log-log} we show the graph of $H_t(x,x)$ as a function of $t$ for a small sample of points $x$ on a log-log scale. Here and elsewhere we focus mainly on the cells $x$ bordering the largest removed square, since these are relatively far from the outer boundary. We take the approximate slope of the portion of the graph that appears close to linear to estimate $-\beta(x)$. In Table~\ref{table:heat-kernel-dirichlet-diagonal-slope-lists} we list these values for the aforementioned cells.
\begin{figure}
	\centering
	\newcommand{\picturewidth}{9cm}
	\makebox[\linewidth]{
		\includegraphics[width=\picturewidth, clip=true, trim=0cm 0cm 1cm 0cm]{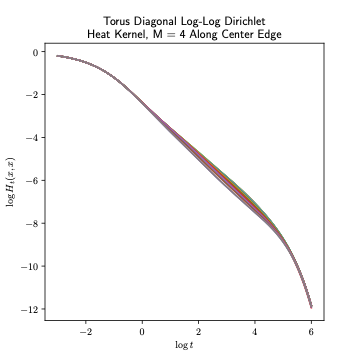}
		\includegraphics[width=\picturewidth, clip=true, trim=0cm 0cm 1cm 0cm]{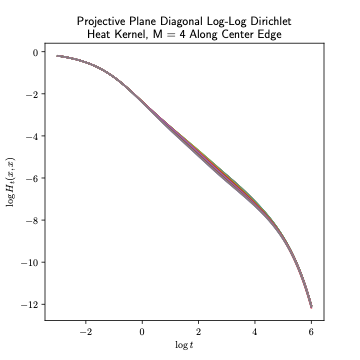}
	}
	\let \picturewidth \relax
	\caption{Log-log plot of the diagonal Dirichlet heat kernel for torus and projective identifications. (Klein identifications are similar.) Each plot is for a cell $x$ bordering the interior square.}
	\label{fig:heat-kernel-diagonal-log-log}
\end{figure}
\begin{table}
	\makebox[\linewidth]{
		\tt
		\begin{tabular}{| lcccc | lcc |}
		\hline
		\multicolumn{8}{| c |}{\rm \bf Slope According to Identification Type} \\ \hline
		{} & {\rm Cell} & {\rm Torus} & {\rm Projective} & {\rm Klein} & {} & {\rm Cell} & {\rm Klein} \\ \hline
		0  & (1000,0222) & -1.301726 & -1.263855 & -1.304539 & 28 & (2000,1000) & -1.304547 \\
		1  & (1001,0222) & -1.253730 & -1.239121 & -1.257284 & 29 & (2000,1001) & -1.256358 \\
		2  & (1002,0222) & -1.221663 & -1.216606 & -1.225863 & 30 & (2000,1002) & -1.223405 \\
		3  & (1010,0222) & -1.202318 & -1.202270 & -1.206413 & 31 & (2000,1010) & -1.203137 \\
		4  & (1011,0222) & -1.198272 & -1.199859 & -1.201669 & 32 & (2000,1011) & -1.198634 \\
		5  & (1012,0222) & -1.183683 & -1.187111 & -1.187614 & 33 & (2000,1012) & -1.184506 \\
		6  & (1020,0222) & -1.174705 & -1.179356 & -1.178552 & 34 & (2000,1020) & -1.176441 \\
		7  & (1021,0222) & -1.171706 & -1.176650 & -1.174911 & 35 & (2000,1021) & -1.174205 \\
		8  & (1022,0222) & -1.164308 & -1.168956 & -1.166984 & 36 & (2000,1022) & -1.166854 \\
		9  & (1100,0222) & -1.163814 & -1.168638 & -1.166578 & 37 & (2000,1100) & -1.166354 \\ 
		10 & (1101,0222) & -1.170175 & -1.175674 & -1.173650 & 38 & (2000,1101) & -1.172659 \\
		11 & (1102,0222) & -1.171464 & -1.177031 & -1.175646 & 39 & (2000,1102) & -1.173191 \\
		12 & (1110,0222) & -1.177146 & -1.182105 & -1.181406 & 40 & (2000,1110) & -1.177974 \\
		13 & (1111,0222) & -1.186926 & -1.190956 & -1.190607 & 41 & (2000,1111) & -1.187285 \\
		14 & (1112,0222) & -1.177146 & -1.182105 & -1.181406 & 42 & (2000,1112) & -1.177974 \\
		15 & (1120,0222) & -1.171464 & -1.177031 & -1.175646 & 43 & (2000,1120) & -1.173191 \\
		16 & (1121,0222) & -1.170175 & -1.175674 & -1.173650 & 44 & (2000,1121) & -1.172659 \\
		17 & (1122,0222) & -1.163814 & -1.168638 & -1.166578 & 45 & (2000,1122) & -1.166354 \\
		18 & (1200,0222) & -1.164308 & -1.168956 & -1.166984 & 46 & (2000,1200) & -1.166854 \\
		19 & (1201,0222) & -1.171706 & -1.176650 & -1.174911 & 47 & (2000,1201) & -1.174205 \\
		20 & (1202,0222) & -1.174705 & -1.179356 & -1.178552 & 48 & (2000,1202) & -1.176441 \\
		21 & (1210,0222) & -1.183683 & -1.187111 & -1.187614 & 49 & (2000,1210) & -1.184506 \\
		22 & (1211,0222) & -1.198272 & -1.199859 & -1.201669 & 50 & (2000,1211) & -1.198634 \\
		23 & (1212,0222) & -1.202318 & -1.202270 & -1.206413 & 51 & (2000,1212) & -1.203137 \\
		24 & (1220,0222) & -1.221663 & -1.216606 & -1.225863 & 52 & (2000,1220) & -1.223405 \\
		25 & (1221,0222) & -1.253730 & -1.239121 & -1.257284 & 53 & (2000,1221) & -1.256358 \\
		26 & (1222,0222) & -1.301726 & -1.263855 & -1.304539 & 54 & (2000,1222) & -1.304547 \\
		27 & (2000,0222) & -1.301722 & -1.263856 & -1.304540 & 55 & (2000,2000) & -1.304540 \\
		\hline
		\end{tabular}
	} \\[0.2cm]
	\caption{Approximate slopes of the log-log Dirichlet heat kernel, calculated along the region $-1 \le \log t \le 3$. (Note that the heat kernel is nonlinear.) The cells listed are those along the edge of the central, removed square. There are twice as many for Klein horizontal identifications because of the asymmetry in that case.}
	\label{table:heat-kernel-dirichlet-diagonal-slope-lists}
\end{table}
\begin{figure}
	\centering
	\newcommand{\picturewidth}{7.5cm}
	\makebox[\linewidth]{
		\includegraphics[width=\picturewidth, clip=true, trim=0cm 0cm 0.5cm 0cm]{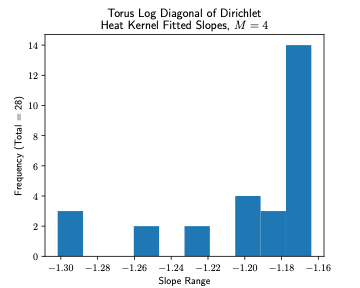}
		\includegraphics[width=\picturewidth, clip=true, trim=0cm 0cm 0.5cm 0cm]{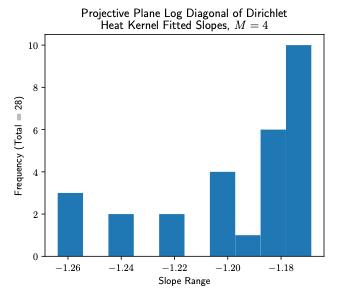}
	}
	\includegraphics[width=\picturewidth, clip=true, trim=0cm 0cm 0.5cm 0cm]{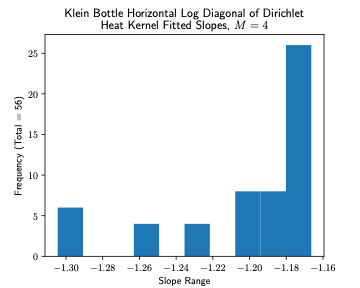}
	\let \picturewidth \relax
	\caption{Histograms for slope values of Table~\ref{table:heat-kernel-dirichlet-diagonal-slope-lists}.}
	\label{fig:heat-kernel-dirichlet-diagonal-slope-histograms}
\end{figure}
In Figure~\ref{fig:heat-kernel-dirichlet-diagonal-slope-histograms} we show histograms of these values. This supplies evidence that $IMC$ is very inhomogeneous. It is not clear whether or not the gaps in the histograms would persist if we were able to extend the computation to higher values of $m$.

Next we consider the off-diagonal behavior of the heat kernel. We fix $y$ and $t$, and examine the graph of $x \mapsto H_t(x, y)$. In Figure~\ref{fig:Dirichlet-heat-kernel-full-plots} we show some samples. As expected we see rapid decay as $x$ moves away from $y$. To see this more clearly we look at the restriction of $x$ to a line segment passing through $y$ in Figure~\ref{fig:Dirichlet-heat-kernel-line-2000-2000-horizontal}. We have also graphed scatter plots of the values of $H_t(x,y)$ for all $x$ of distance $k$ to $y$ as $k$ varies (again, log-log plots), as shown in Figure~\ref{fig:Dirichlet-heat-kernel-scatter-log-torus}. From this we obtain conjectural bounds
\[ c_1 e^{-c_2 d(x, y)^\gamma} \le H_t(x, y) \le c_1' e^{-c_2' d(x, y)^{\gamma'}}, \]
where $c_1, c_2, c_1', c_2'$ depend on $y$ and $t$ and
\[ a \le \gamma \le b \qqand a' \le \gamma' \le b'. \]

\newcommand{\picturewidth}{6cm}
\begin{figure}
	\centering
	
	\makebox[\linewidth]{
		\includegraphics[width=\picturewidth, clip=true, trim=0.5cm 0.5cm 0.25cm 0cm]{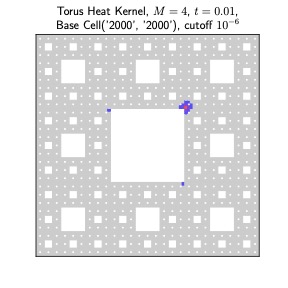}
		\includegraphics[width=\picturewidth, clip=true, trim=0.5cm 0.5cm 0.25cm 0cm]{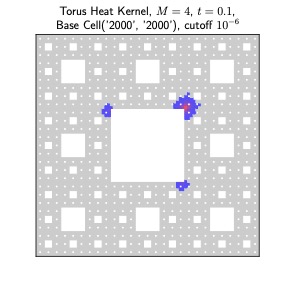}
		\includegraphics[width=\picturewidth, clip=true, trim=0.5cm 0.5cm 0.25cm 0cm]{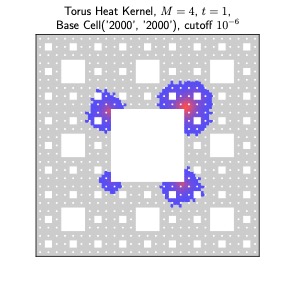}
	}
	
	\makebox[\linewidth]{
		\includegraphics[width=\picturewidth, clip=true, trim=0.5cm 0.5cm 0.25cm 0cm]{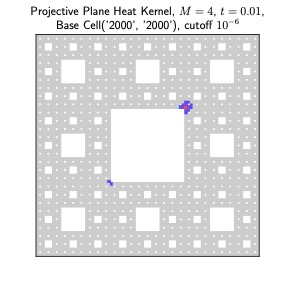}
		\includegraphics[width=\picturewidth, clip=true, trim=0.5cm 0.5cm 0.25cm 0cm]{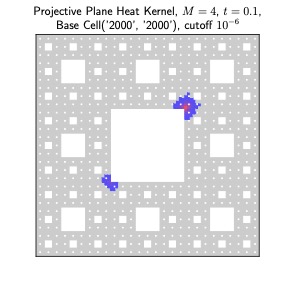}
		\includegraphics[width=\picturewidth, clip=true, trim=0.5cm 0.5cm 0.25cm 0cm]{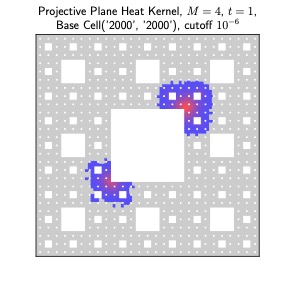}
	}
	
	\makebox[\linewidth]{
		\includegraphics[width=\picturewidth, clip=true, trim=0.5cm 0.5cm 0.25cm 0cm]{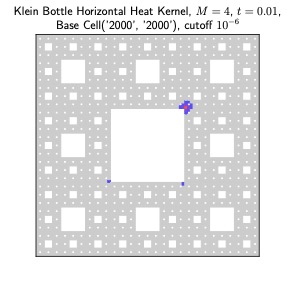}
		\includegraphics[width=\picturewidth, clip=true, trim=0.5cm 0.5cm 0.25cm 0cm]{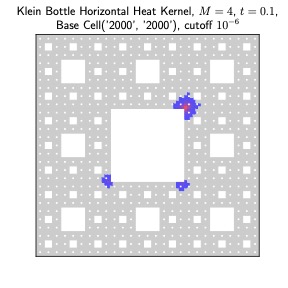}
		\includegraphics[width=\picturewidth, clip=true, trim=0.5cm 0.5cm 0.25cm 0cm]{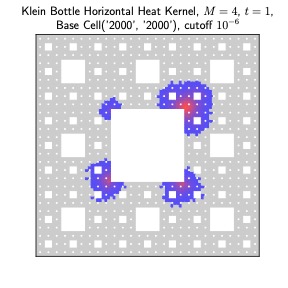}
	}
	\caption{The off-diagonal heat kernel for torus (top), projective, and Klein horizontal (bottom) identifications at times $t = 0.01$ (left), $0.1$, and $1$ (right). Values with magnitude $< 10^{-6}$ are gray.}
	\label{fig:Dirichlet-heat-kernel-full-plots}
\end{figure}
\let \picturewidth \relax
\newcommand{\picturewidth}{8cm}
\begin{figure}
	\centering
	\makebox[\linewidth]{
		\includegraphics[width=\picturewidth, clip=true, trim=0.5cm 0cm 1cm 0cm]{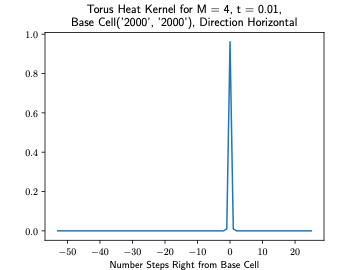}
		\includegraphics[width=\picturewidth, clip=true, trim=0.5cm 0cm 1cm 0cm]{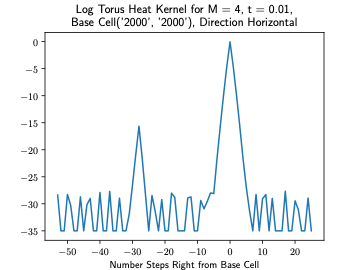}
	}
	\makebox[\linewidth]{
		\includegraphics[width=\picturewidth, clip=true, trim=0.5cm 0cm 1cm 0cm]{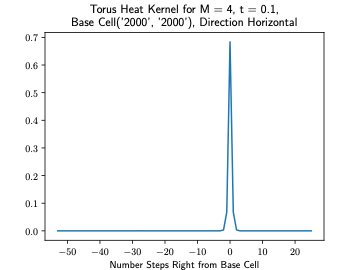}
		\includegraphics[width=\picturewidth, clip=true, trim=0.5cm 0cm 1cm 0cm]{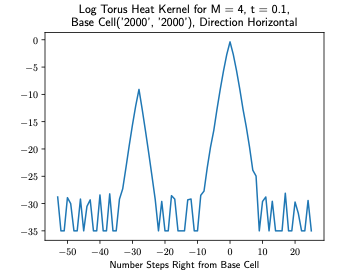}
	}
	\makebox[\linewidth]{
		\includegraphics[width=\picturewidth, clip=true, trim=0.5cm 0cm 1cm 0cm]{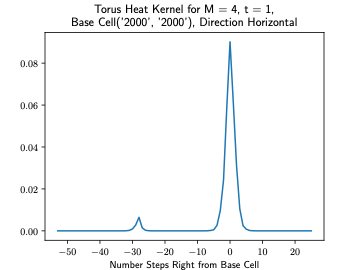}
		\includegraphics[width=\picturewidth, clip=true, trim=0.5cm 0cm 1cm 0cm]{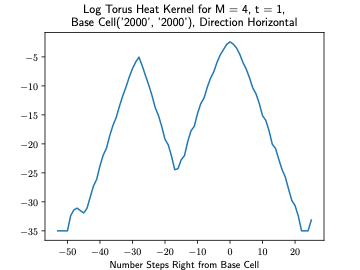}
	}
	\caption{With torus identifications, the Dirichlet heat kernel along a horizontal line (left) and its log (right) at times $t = 0.01$ (top), $0.1$, and $1$ (bottom). Log values $\le -30$ are partly numerical error.}
	\label{fig:Dirichlet-heat-kernel-line-2000-2000-horizontal}
\end{figure}
\let \picturewidth \relax
\newcommand{\picturewidth}{6cm}
\begin{figure}
	\centering
	\makebox[\linewidth]{
		\includegraphics[width=\picturewidth, clip=true, trim=0.5cm 0cm 0cm 0cm]{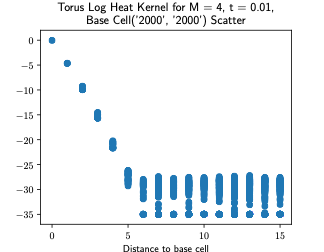}
		\includegraphics[width=\picturewidth, clip=true, trim=0.5cm 0cm 0cm 0cm]{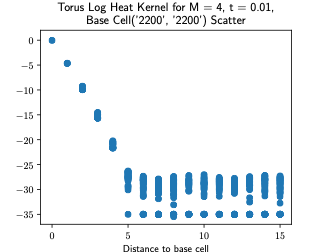}
		\includegraphics[width=\picturewidth, clip=true, trim=0.5cm 0cm 0cm 0cm]{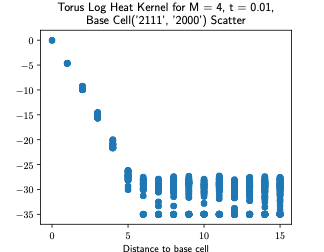}
	}
	
	\vspace{0.2cm}
	
	\makebox[\linewidth]{
		\includegraphics[width=\picturewidth, clip=true, trim=0.5cm 0cm 0cm 0cm]{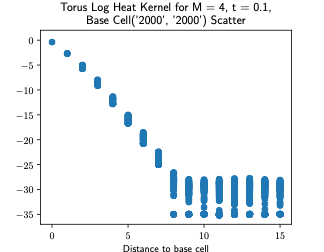}
		\includegraphics[width=\picturewidth, clip=true, trim=0.5cm 0cm 0cm 0cm]{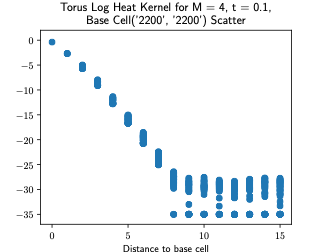}
		\includegraphics[width=\picturewidth, clip=true, trim=0.5cm 0cm 0cm 0cm]{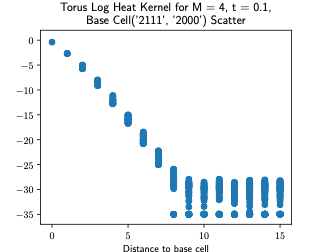}
	}
	
	\vspace{0.2cm}
	
	\makebox[\linewidth]{
		\includegraphics[width=\picturewidth, clip=true, trim=0.5cm 0cm 0cm 0cm]{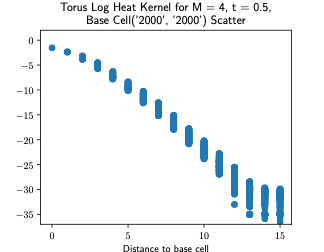}
		\includegraphics[width=\picturewidth, clip=true, trim=0.5cm 0cm 0cm 0cm]{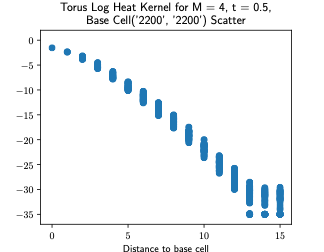}
		\includegraphics[width=\picturewidth, clip=true, trim=0.5cm 0cm 0cm 0cm]{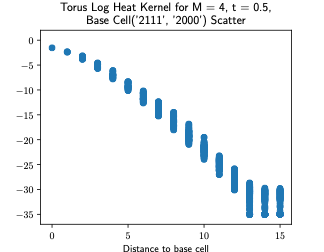}
	}
	
	\vspace{0.2cm}
	
	\makebox[\linewidth]{
		\includegraphics[width=\picturewidth, clip=true, trim=0.5cm 0cm 0cm 0cm]{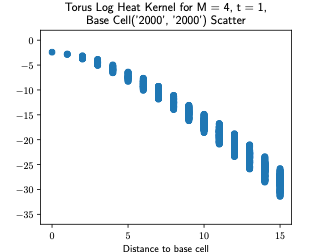}
		\includegraphics[width=\picturewidth, clip=true, trim=0.5cm 0cm 0cm 0cm]{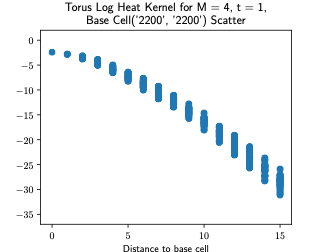}
		\includegraphics[width=\picturewidth, clip=true, trim=0.5cm 0cm 0cm 0cm]{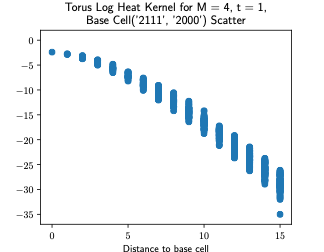}
	}
	\caption{Log Dirichlet heat kernel plotted against distance to base cell, for three base cells (columns), times $t = 0.01$ (top), $0.1$, $0.5$, and $1$ (bottom), and torus identifications. (Projective and Klein yield similar plots.)}
	\label{fig:Dirichlet-heat-kernel-scatter-log-torus}
\end{figure}
\let \picturewidth \relax


\section{The Wave Propagator on \textit{IMC}}
\label{sec: wave}

In Figures~\ref{fig:wave-prop-base2000-2000}--\ref{fig:wave-prop-base2000-1111} we show the graphs of the wave propagator~\eqref{eq:wave-propagator} as a function of $x$ for $m = 4$ and three different choices of $y$ with $t = 1, 2, 3, 4$. These are all shown with torus identifications.
\newcommand{\picturewidth}{6cm}
\begin{figure}
	\makebox[\textwidth]{
		\includegraphics[width=\picturewidth, clip=true, trim=0.5cm 0.5cm 0.5cm 0cm]{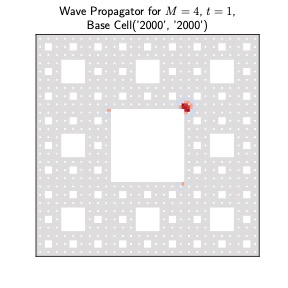}
		\includegraphics[width=\picturewidth, clip=true, trim=0.5cm 0.5cm 0.5cm 0cm]{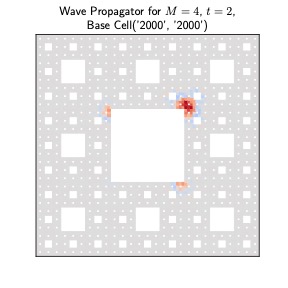}
		\includegraphics[width=\picturewidth, clip=true, trim=0.5cm 0.5cm 0.5cm 0cm]{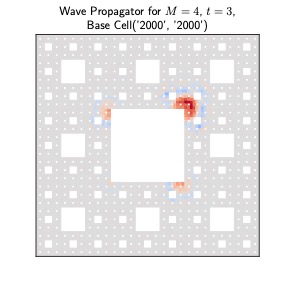}
	}
	
	\makebox[\textwidth]{
		\includegraphics[width=\picturewidth, clip=true, trim=0.5cm 0.5cm 0.5cm 0cm]{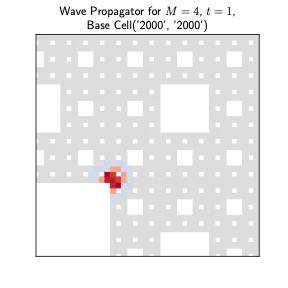}
		\includegraphics[width=\picturewidth, clip=true, trim=0.5cm 0.5cm 0.5cm 0cm]{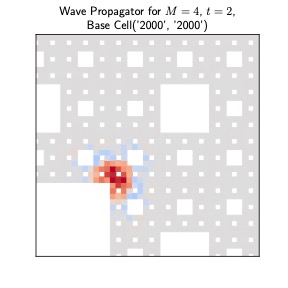}
		\includegraphics[width=\picturewidth, clip=true, trim=0.5cm 0.5cm 0.5cm 0cm]{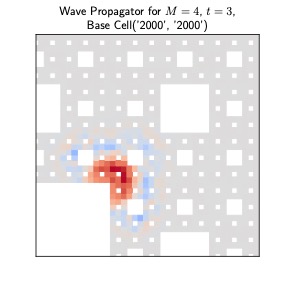}
	}
	
	\includegraphics[width=\picturewidth, clip=true, trim=0.5cm 0.5cm 0.5cm 0cm]{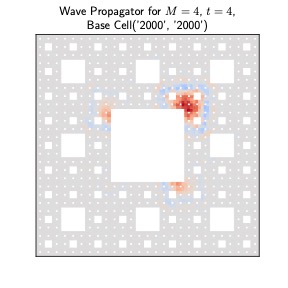}
	\includegraphics[width=\picturewidth, clip=true, trim=0.5cm 0.5cm 0.5cm 0cm]{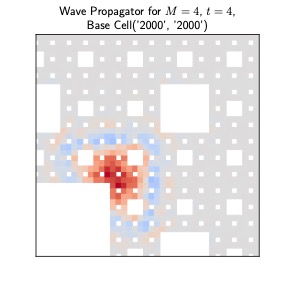}
	\caption{The wave propagator with torus identifications for one base cell is shown at times $t = 1, 2, 3$ (across the top) and $t = 4$ (bottom left). Below $t=1,2,3$ and beside $t=4$ are zoomed pictures showing some of the more prominent behavior.}
	\label{fig:wave-prop-base2000-2000}
\end{figure}
\begin{figure}
	\makebox[\textwidth]{
		\includegraphics[width=\picturewidth, clip=true, trim=0.5cm 0.5cm 0.5cm 0cm]{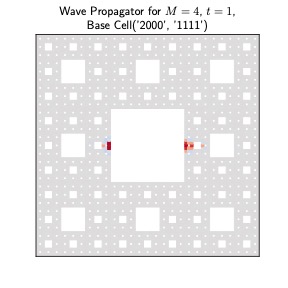}
		\includegraphics[width=\picturewidth, clip=true, trim=0.5cm 0.5cm 0.5cm 0cm]{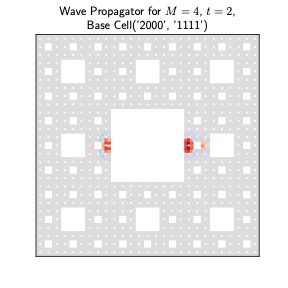}
		\includegraphics[width=\picturewidth, clip=true, trim=0.5cm 0.5cm 0.5cm 0cm]{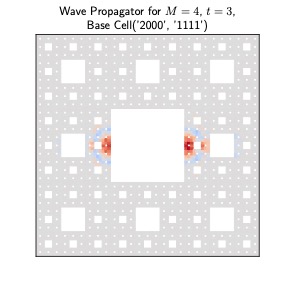}
	}
	
	\makebox[\textwidth]{
		\includegraphics[width=\picturewidth, clip=true, trim=0.5cm 0.5cm 0.5cm 0cm]{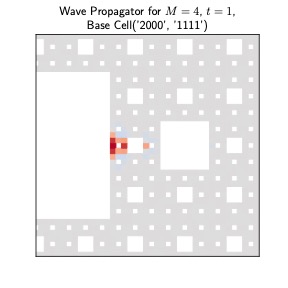}
		\includegraphics[width=\picturewidth, clip=true, trim=0.5cm 0.5cm 0.5cm 0cm]{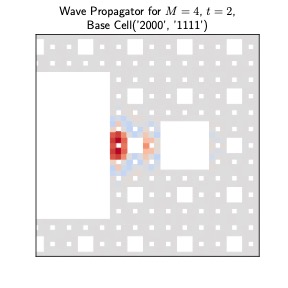}
		\includegraphics[width=\picturewidth, clip=true, trim=0.5cm 0.5cm 0.5cm 0cm]{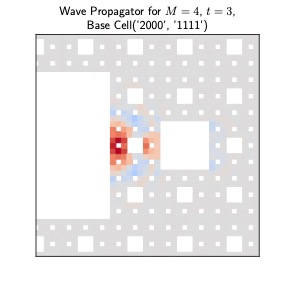}
	}
	
	\includegraphics[width=\picturewidth, clip=true, trim=0.5cm 0.5cm 0.5cm 0cm]{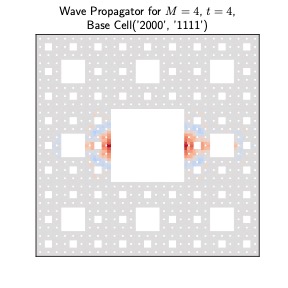}
	\includegraphics[width=\picturewidth, clip=true, trim=0.5cm 0.5cm 0.5cm 0cm]{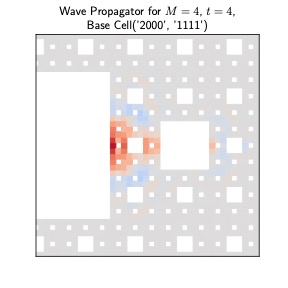}
	\caption{Refer to the description of Figure~\ref{fig:wave-prop-base2000-2000}.}
	\label{fig:wave-prop-base2000-1111}
\end{figure}
\let \picturewidth \relax
Note that we do not expect a finite propagation speed, since we are working on a graph (see also \cite{lee} for non-finite propagation speed on other fractals), but we do expect most of the significant support to be centered at $y$ and to increase with $t$. We clearly see both behaviors. Comparing $t = 1$ with $t = 2$ and $t = 2$ with $t = 4$ we see a qualitative spreading of the size of the significant support, although we do not see how to make this into a quantitative statement.

It appears that the maximum value occurs near $x = y$ but not always at $x = y$, and the propagator appears to be bounded. This is in contrast to the propagator in $\RR^2$ that has a singularity at $|x - y| = t$.

In Figures~\ref{fig:wave-prop-scatters-t1-t4}--\ref{fig:wave-prop-scatters-t5-t8} we show scatter plots of wave propagator values for cells near $x$.

\begin{figure}
	\centering
	\newcommand{\picturewidth}{6cm}
	\makebox[\linewidth]{
		\includegraphics[width=\picturewidth]{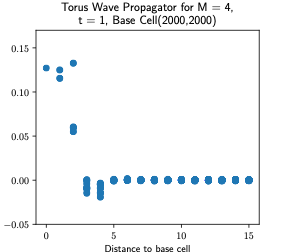}
		\includegraphics[width=\picturewidth]{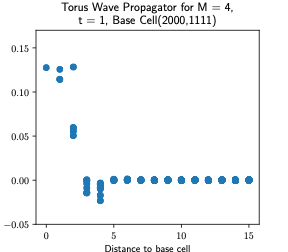}
	}
	\makebox[\linewidth]{
		\includegraphics[width=\picturewidth]{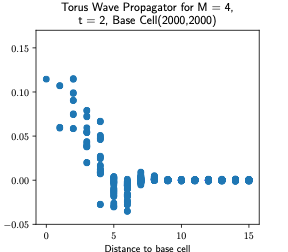}
		\includegraphics[width=\picturewidth]{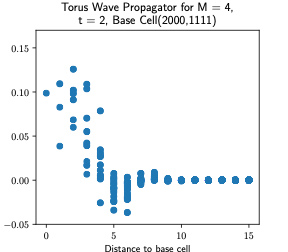}
	}
	\makebox[\linewidth]{
		\includegraphics[width=\picturewidth]{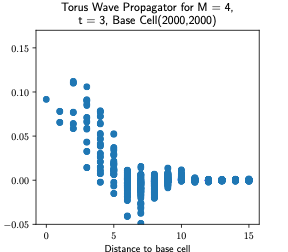}
		\includegraphics[width=\picturewidth]{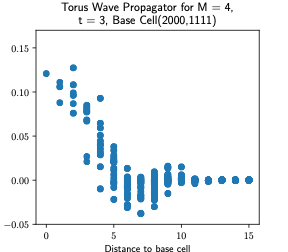}
	}
	\makebox[\linewidth]{
		\includegraphics[width=\picturewidth]{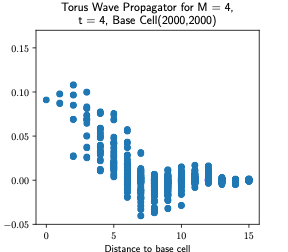}
		\includegraphics[width=\picturewidth]{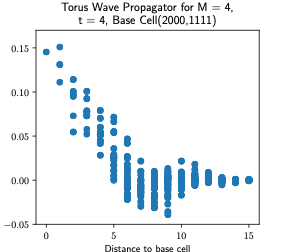}
	}
	\let \picturewidth \relax
	\caption{Scatter plots of the wave propagator. Columns (left to right) correspond to the base cells in Figures~\ref{fig:wave-prop-base2000-2000} and \ref{fig:wave-prop-base2000-1111}. Rows are times $t = 1, 2, 3, 4$.}
	\label{fig:wave-prop-scatters-t1-t4}
\end{figure}
\begin{figure}
	\centering
	\newcommand{\picturewidth}{6cm}
	\makebox[\linewidth]{
		\includegraphics[width=\picturewidth]{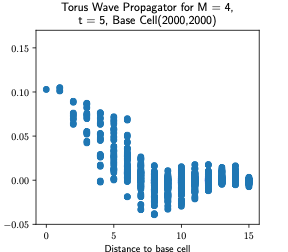}
		\includegraphics[width=\picturewidth]{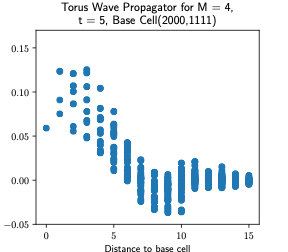}
	}
	\makebox[\linewidth]{
		\includegraphics[width=\picturewidth]{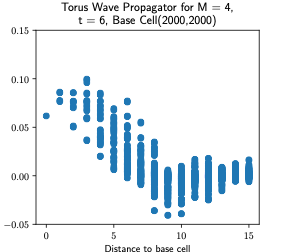}
		\includegraphics[width=\picturewidth]{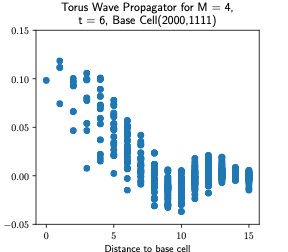}
	}
	\makebox[\linewidth]{
		\includegraphics[width=\picturewidth]{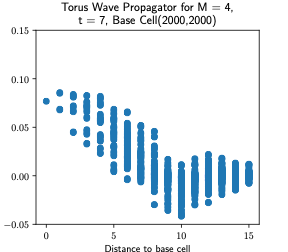}
		\includegraphics[width=\picturewidth]{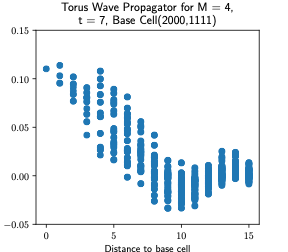}
	}
	\makebox[\linewidth]{
		\includegraphics[width=\picturewidth]{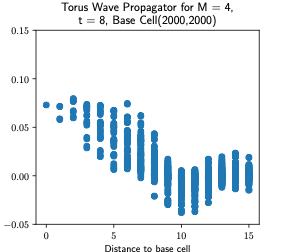}
		\includegraphics[width=\picturewidth]{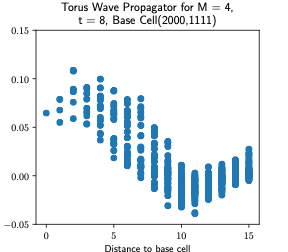}
	}
	\let \picturewidth \relax
	\caption{As in Figure~\ref{fig:wave-prop-scatters-t1-t4}, but with times $t = 5, 6, 7, 8$.}
	\label{fig:wave-prop-scatters-t5-t8}
\end{figure}


\section{Harmonic Functions}
\label{sec: harmonic}

Harmonic functions on $IMC$ are characterized by the property that the value at any cell is equal to the average value on the four neighboring cells. We expect that the space of harmonic functions is infinite-dimensional, so in particular there are nonzero harmonic functions that are close to zero on any of the approximations $\widetilde{MC}_m$. Thus we cannot learn very much about the full space of harmonic functions on $IMC$ by studying harmonic functions on $\widetilde{MC}_m$. Nevertheless, it is interesting to study the analog of the Poisson kernel on $\widetilde{MC}_m$.

For this purpose we define the boundary of $\widetilde{MC}_m$ to be the $4 \left( 3^m - 1 \right)$ cells along the boundary of the square containing $\widetilde{MC}_m$, and everything else forms the interior. We impose the harmonic condition only on interior cells, and prescribe values on the boundary cells. The Poisson kernel $P(x, y)$ is the function that provides the interior values in terms of the boundary values:
\begin{align}
	\label{eq: harmonic interior}
	h(x) = \sum_{y \in \partial \widetilde{MC}_m} P(x, y) h(y).
\end{align}
It follows from general principles that a unique such function exists and is nonnegative. In fact $x \mapsto P(x, y)$ is the unique harmonic function satisfying $P(x, z) = \delta_{xz}$ for $z$ in the boundary. So $P(y, y) = 1$ and $P(x, y)$ is expected to decay as $x$ moves away from $y$, but not as rapidly as the heat kernel.

In Figure~\ref{fig:harmonics-sampling} we show graphs of $x \mapsto P(x, y)$ for a sampling of points $y$, all with torus identifications. We also show scatter plots of the values of $P(x, y)$ as $x$ varies over cells of distance $k$ to $y$.

\begin{figure}
	\makebox[\linewidth]{
		\includegraphics[width=5cm, clip=true, trim=1cm 1cm 0.75cm 0cm]{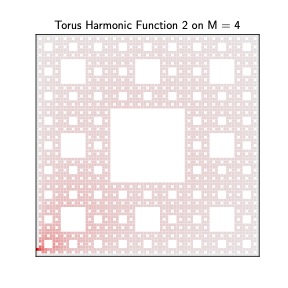}
		\hspace{0.9cm}
		\includegraphics[width=5cm, clip=true, trim=1cm 1cm 0.75cm 0cm]{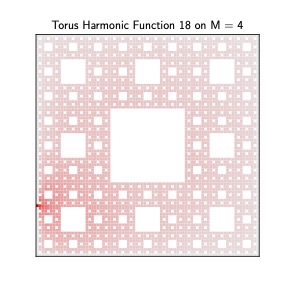}
		\hspace{0.9cm}
		\includegraphics[width=5cm, clip=true, trim=1cm 1cm 0.75cm 0cm]{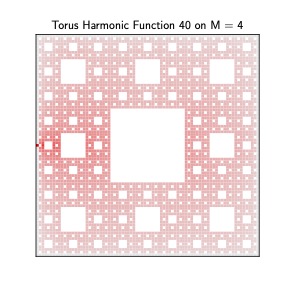}
	}
	\makebox[\linewidth]{
		\includegraphics[width=6cm, clip=true, trim=0.2cm 0.2cm 0.2cm 0.2cm]{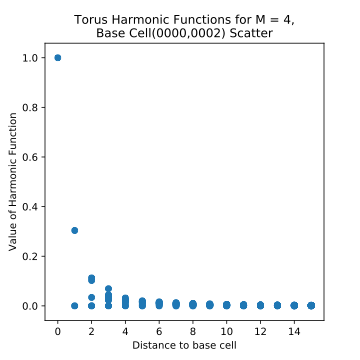}
		\includegraphics[width=6cm, clip=true, trim=0.2cm 0.2cm 0.2cm 0.2cm]{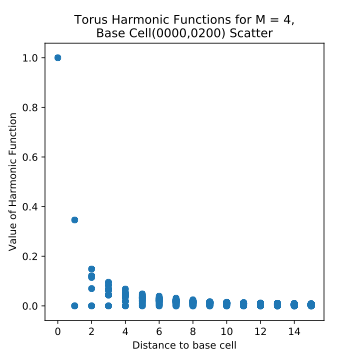}
		\includegraphics[width=6cm, clip=true, trim=0.2cm 0.2cm 0.2cm 0.2cm]{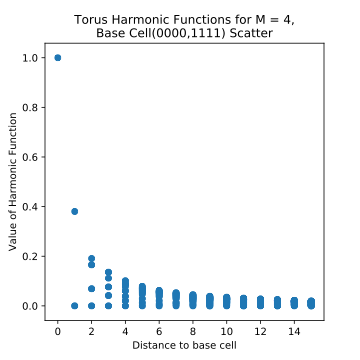}
	}
	\makebox[\linewidth]{
		\includegraphics[width=6cm, clip=true, trim=0cm 0.2cm 0.2cm 0.2cm]{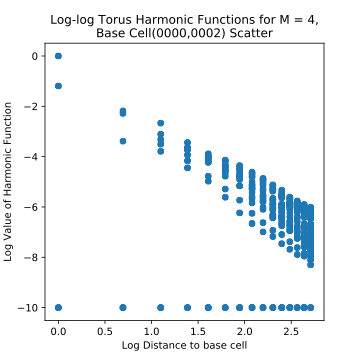}
		\includegraphics[width=6cm, clip=true, trim=0cm 0.2cm 0.2cm 0.2cm]{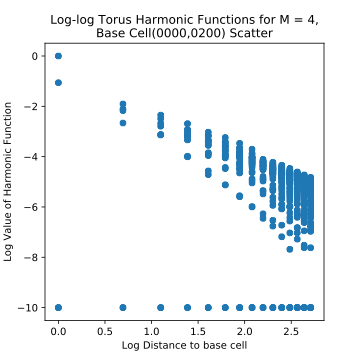}
		\includegraphics[width=6cm, clip=true, trim=0cm 0.2cm 0.2cm 0.2cm]{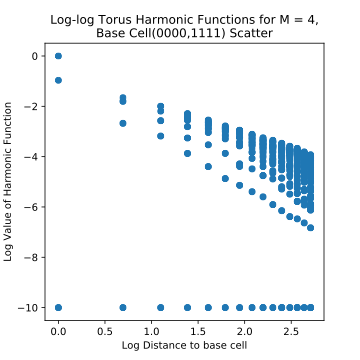}
	}
	\caption{(Top Row) A sampling of $x \mapsto P(x,y)$ on $M=4$, with $y$ chosen to be three cells along the boundary.
	(Middle Row) For the same points $y$ as above, the values of $x \mapsto P(x,y)$ plotted against $d(x, y) \le 15$.
	(Bottom Row) Log-log versions of the plots shown in the middle row. Boundary cells are omitted in these plots.}
	\label{fig:harmonics-sampling}
\end{figure}


\section{Spectrum of the Laplacian on Magic Carpet Fractals}
\label{sec: uniform}

Let $MC_m$ denote the level-$m$ approximation with the outer boundary identified in the same manner as the inner boundaries (so we have three versions: $MCT_m$, $MCK_m$, and $MCP_m$). So $MC_m$ has $8^m$ cells; we assign to each of them measure $1/8^m$, and we take the side lengths to be $1/3^m$, so each $MC_{m+1}$ is a refinement of $MC_m$ in the appropriate sense. In the limit as $m \to \infty$ we obtain a magic carpet fractal $MC$. The Laplacian on $MC_m$ is given by 
\begin{align}
	-\D^{(m)} f(x) = \sum_{y \underset m \sim x} \left(f(x) - f(y) \right).
\end{align} 
It is a symmetric operator with $8^m$ eigenfunctions with eigenvalues in $[0, 8]$ (by Ger\v{s}gorin's circle theorem). We would like to claim the existence of a limiting operator
\begin{align}
	-\D f = \lim_{m \to \infty} R^m \D^{(m)} f
\end{align}
on $MC$, for the appropriate renormalization factor $R$. Then the spectrum of $-\D$ would be the limit of the spectrum of $\D^{(m)}$ multiplied by $R^m$. In fact, there is no published proof of the existence of this limit, but numerical data in this paper and in previous works (\cite{bello-li-strichartz}, \cite{molitor-ott-strichartz} in the torus identification case) leaves little doubt that the limit exists.

By computing the spectra for $m = 2, 3, 4$ and taking ratios we can estimate the renormalization factor $R$. This data is shown in Tables~\ref{table:eigenvalue-ratios-tg}--\ref{table:eigenvalue-ratios-kbhg} for the beginning of the spectrum, and the remainder can be found on the website \cite{website}.
\begin{table}
	\centering
	\makebox[\linewidth]{
		\tt
		\begin{tabular}{| l l l l l l |}
		\hline
		\multicolumn{6}{| c |}{\rm \bf Torus Glued Identifications} \\
		{} & \parbox{2.5cm}{\rm \raggedright Eigenvalue $\lambda_2$ for $M=2$} & \parbox{2.5cm}{\rm \raggedright Eigenvalue $\lambda_3$ for $M=3$} & \parbox{2.5cm}{\rm \raggedright Eigenvalue $\lambda_4$ for $M=4$} & \parbox{2.5cm}{\rm \raggedright Ratio $\lambda_2 / \lambda_3$ }& \parbox{2.5cm}{\rm \raggedright Ratio $\lambda_3 / \lambda_4$} \\[0.25cm] \hline
		0 & \hphantom{** }0.0 & \hphantom{** }0.0 & \hphantom{** }-0.0 & - & - \\
		1 & \hphantom{** }0.4410218 & \hphantom{** }0.0718171 & \hphantom{** }0.0110916 & 6.1408993 & 6.4749219 \\
		2 & ** 0.690591 & ** 0.1119098 & ** 0.0173466 & 6.1709581 & 6.4513906 \\
		3 & ** 0.690591 & ** 0.1119098 & ** 0.0173466 & 6.1709581 & 6.4513906 \\
		4 & \hphantom{** }0.7587998 & \hphantom{** }0.1205024 & \hphantom{** }0.0185273 & 6.2969667 & 6.5040388 \\
		5 & \hphantom{** }1.4269914 & \hphantom{** }0.2334006 & \hphantom{** }0.0359611 & 6.1139159 & 6.4903591 \\
		6 & ** 1.482754 & ** 0.245267 & ** 0.0379436 & 6.0454696 & 6.4639853 \\
		7 & ** 1.482754 & ** 0.245267 & ** 0.0379436 & 6.0454696 & 6.4639853 \\
		8 & \hphantom{** }1.4983881 & \hphantom{** }0.2518546 & \hphantom{** }0.0392556 & 5.949417 & 6.415758 \\
		9 & \hphantom{** }1.5692227 & \hphantom{** }0.2780838 & \hphantom{** }0.0438667 & 5.6429846 & 6.3392857 \\
		10 & \hphantom{** }1.8746245 & \hphantom{** }0.3402416 & \hphantom{** }0.053536 & 5.509687 & 6.3553786 \\
		11 & ** 1.8836369 & ** 0.3455927 & ** 0.0552662 & 5.4504528 & 6.2532401 \\
		12 & ** 1.8836369 & ** 0.3455927 & ** 0.0552662 & 5.4504528 & 6.2532401 \\
		13 & \hphantom{** }2.0 & ** 0.415878 & ** 0.0649221 & 4.8091024 & 6.4058044 \\
		14 & \hphantom{** }2.3293357 & ** 0.415878 & ** 0.0649221 & 5.6010068 & 6.4058044 \\
		15 & ** 2.4155337 & \hphantom{** }0.4189091 & \hphantom{** }0.0658944 & 5.7662475 & 6.3572796 \\
		\hline
		\end{tabular}
	} \\[0.2cm]
	\caption{The beginning eigenvalues with torus glued identifications at levels $m=2,3,4$, and the ratios of these eigenvalues.}
	\label{table:eigenvalue-ratios-tg}
\end{table}
\begin{table}
	\centering
	\makebox[\linewidth]{
		\tt
		\begin{tabular}{| l l l l l l |}
		\hline
		\multicolumn{6}{| c |}{\rm \bf Projective Plane Glued Identifications} \\
		{} & \parbox{2.5cm}{\rm \raggedright Eigenvalue $\lambda_2$ for $M=2$} & \parbox{2.5cm}{\rm \raggedright Eigenvalue $\lambda_3$ for $M=3$} & \parbox{2.5cm}{\rm \raggedright Eigenvalue $\lambda_4$ for $M=4$} & \parbox{2.5cm}{\rm \raggedright Ratio $\lambda_2 / \lambda_3$ }& \parbox{2.5cm}{\rm \raggedright Ratio $\lambda_3 / \lambda_4$} \\[0.25cm] \hline
		0 & \hphantom{** }-0.0 & \hphantom{** }-0.0 & \hphantom{** }-0.0 & - & - \\
		1 & \hphantom{** }0.3058223 & \hphantom{** }0.0477565 & \hphantom{** }0.0074541 & 6.4037878 & 6.406704 \\
		2 & \hphantom{** }0.4410218 & \hphantom{** }0.0729375 & \hphantom{** }0.0115116 & 6.046572 & 6.3360045 \\
		3 & \hphantom{** }0.7587998 & \hphantom{** }0.1233985 & \hphantom{** }0.0194031 & 6.1491843 & 6.3597326 \\
		4 & ** 1.1250751 & ** 0.1858666 & ** 0.0292749 & 6.0531309 & 6.3490018 \\
		5 & ** 1.1250751 & ** 0.1858666 & ** 0.0292749 & 6.0531309 & 6.3490018 \\
		6 & \hphantom{** }1.3324988 & \hphantom{** }0.2338597 & \hphantom{** }0.0373292 & 5.6978556 & 6.2647901 \\
		7 & \hphantom{** }1.3652037 & \hphantom{** }0.2388857 & \hphantom{** }0.0380849 & 5.7148832 & 6.2724489 \\
		8 & \hphantom{** }1.4983881 & \hphantom{** }0.254542 & \hphantom{** }0.0403709 & 5.8866048 & 6.3050832 \\
		9 & \hphantom{** }1.5692227 & \hphantom{** }0.2898892 & \hphantom{** }0.0466163 & 5.4131807 & 6.2186257 \\
		10 & \hphantom{** }1.8746245 & \hphantom{** }0.3058223 & \hphantom{** }0.0477565 & 6.1297834 & 6.4037878 \\
		11 & ** 2.0 & \hphantom{** }0.3422554 & ** 0.0541734 & 5.8435898 & 6.3177806 \\
		12 & ** 2.0 & ** 0.343133 & ** 0.0541734 & 5.8286443 & 6.3339803 \\
		13 & ** 2.0371299 & ** 0.343133 & \hphantom{** }0.0548889 & 5.9368529 & 6.2514145 \\
		14 & ** 2.0371299 & \hphantom{** }0.4275677 & \hphantom{** }0.0690194 & 4.7644621 & 6.1948905 \\
		15 & \hphantom{** }2.1109942 & \hphantom{** }0.430103 & \hphantom{** }0.0697148 & 4.9081133 & 6.1694614 \\
		\hline
		\end{tabular}
	} \\[0.2cm]
	\caption{As in Table~\ref{table:eigenvalue-ratios-tg}, but the identifications here are projective.}
	\label{table:eigenvalue-ratios-ppg}
\end{table}
\begin{table}
	\centering
	\makebox[\linewidth]{
		\tt
		\begin{tabular}{| l l l l l l |}
		\hline
		\multicolumn{6}{| c |}{\rm \bf Klein Bottle Horizontal Glued Identifications} \\
		{} & \parbox{2.5cm}{\rm \raggedright Eigenvalue $\lambda_2$ for $M=2$} & \parbox{2.5cm}{\rm \raggedright Eigenvalue $\lambda_3$ for $M=3$} & \parbox{2.5cm}{\rm \raggedright Eigenvalue $\lambda_4$ for $M=4$} & \parbox{2.5cm}{\rm \raggedright Ratio $\lambda_2 / \lambda_3$ }& \parbox{2.5cm}{\rm \raggedright Ratio $\lambda_3 / \lambda_4$} \\[0.25cm] \hline
		0 & \hphantom{** }0.0 & \hphantom{** }0.0 & \hphantom{** }-0.0 & - & - \\
		1 & \hphantom{** }0.4410218 & \hphantom{** }0.0723638 & \hphantom{** }0.0112916 & 6.0945093 & 6.408612 \\
		2 & \hphantom{** }0.690591 & \hphantom{** }0.1119495 & \hphantom{** }0.0173792 & 6.1687743 & 6.4415609 \\
		3 & \hphantom{** }0.757329 & \hphantom{** }0.121964 & \hphantom{** }0.018978 & 6.2094458 & 6.4266102 \\
		4 & \hphantom{** }0.7587998 & \hphantom{** }0.123382 & \hphantom{** }0.0193997 & 6.1500041 & 6.3599916 \\
		5 & \hphantom{** }1.1250751 & \hphantom{** }0.1852829 & \hphantom{** }0.0290511 & 6.0722031 & 6.3778205 \\
		6 & \hphantom{** }1.482754 & \hphantom{** }0.2465753 & \hphantom{** }0.0384654 & 6.0133933 & 6.410307 \\
		7 & \hphantom{** }1.4983881 & \hphantom{** }0.2530836 & \hphantom{** }0.0398039 & 5.920527 & 6.3582652 \\
		8 & \hphantom{** }1.5692227 & \hphantom{** }0.2844123 & \hphantom{** }0.0453381 & 5.5174219 & 6.2731406 \\
		9 & \hphantom{** }1.8662197 & \hphantom{** }0.3196011 & \hphantom{** }0.0507154 & 5.8392155 & 6.3018541 \\
		10 & \hphantom{** }1.8746245 & \hphantom{** }0.3415234 & \hphantom{** }0.0535803 & 5.4890069 & 6.3740525 \\
		11 & \hphantom{** }1.8836369 & \hphantom{** }0.3429855 & \hphantom{** }0.054342 & 5.491885 & 6.3116066 \\
		12 & \hphantom{** }1.9260699 & \hphantom{** }0.3463405 & \hphantom{** }0.055662 & 5.5612036 & 6.2222017 \\
		13 & \hphantom{** }2.0 & \hphantom{** }0.356565 & \hphantom{** }0.0576952 & 5.6090762 & 6.1801455 \\
		14 & \hphantom{** }2.0371299 & \hphantom{** }0.4067644 & \hphantom{** }0.0640321 & 5.0081327 & 6.3525108 \\
		15 & \hphantom{** }2.3293357 & \hphantom{** }0.4245981 & \hphantom{** }0.0675213 & 5.4859772 & 6.2883567 \\
		\hline
		\end{tabular}
	} \\[0.2cm]
	\caption{As in Table~\ref{table:eigenvalue-ratios-tg}, but the identifications here are Klein bottle horizontal.}
	\label{table:eigenvalue-ratios-kbhg}
\end{table}
We notice that the ratios decrease as you move farther up the spectrum, and we believe that computational error degrades the results as the eigenvalues increase, so we take the average of the first ten ratios $\lambda_3/\lambda_4$ from each table to estimate
\begin{align}
	R_T \approx 6.441049, 
	\qquad
	R_K \approx 6.373221, 
	\qqand
	R_P \approx 6.326518. 
\end{align}
These are close but not equal. Note that some of the eigenvalues for the torus and projective identifications have multiplicity two. This is easily explained by the fact that $MCT$ and $MCP$ have a dihedral $D_4$ group of symmetries, and $D_4$ has a two-dimensional irreducible representation. The symmetry group of $MCK$ is $\ZZ_2 \times \ZZ_2$, which is abelian, so it only has one-dimensional irreducible representations. As can be seen in Tables~\ref{table:eigenvalue-ratios-tg} and~\ref{table:eigenvalue-ratios-ppg}, the location in the spectrum of the multiplicity two eigenvalues agrees from one level $m$ to the next only in the bottom portion of the spectrum, so the use of ratios is only meaningful below these points.

In Figure~\ref{fig:glued-eigenvalue-counting-functions} we show the graphs of the eigenvalue counting functions and in Figure~\ref{fig:weyl-standard-glued} the Weyl ratio (log-log).
\begin{figure}
	\newcommand{\picturewidth}{6cm}
	\centering
	\makebox[\linewidth]{
		\includegraphics[width=\picturewidth, clip=true, trim=0cm 0cm 1cm 0cm]{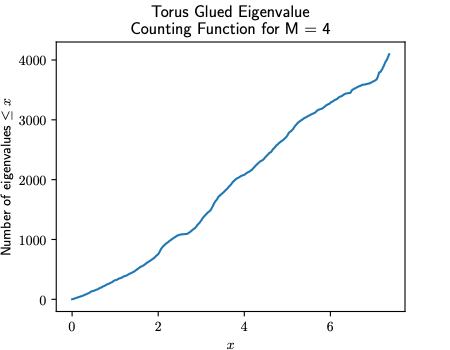}
		\includegraphics[width=\picturewidth, clip=true, trim=0cm 0cm 1cm 0cm]{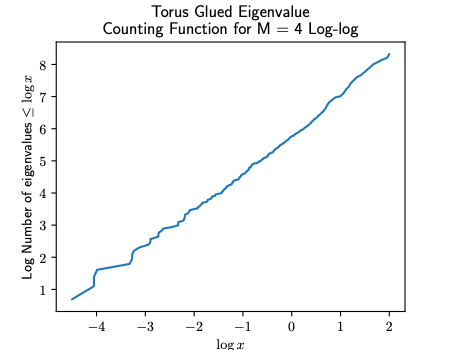}
	}
	\makebox[\linewidth]{
		\includegraphics[width=\picturewidth, clip=true, trim=0cm 0cm 1cm 0cm]{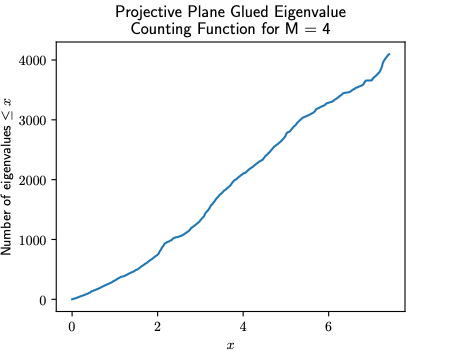}
		\includegraphics[width=\picturewidth, clip=true, trim=0cm 0cm 1cm 0cm]{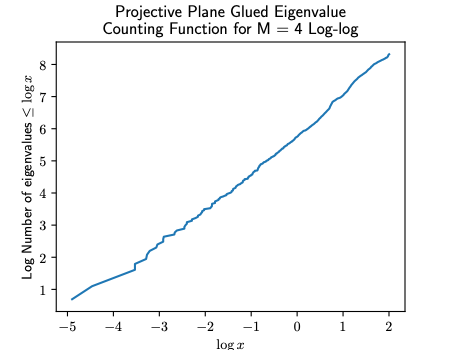}
	}
	\makebox[\linewidth]{
		\includegraphics[width=\picturewidth, clip=true, trim=0cm 0cm 1cm 0cm]{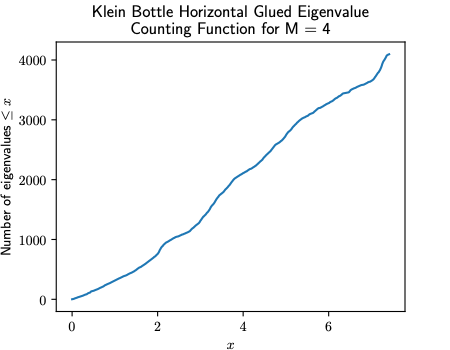}
		\includegraphics[width=\picturewidth, clip=true, trim=0cm 0cm 1cm 0cm]{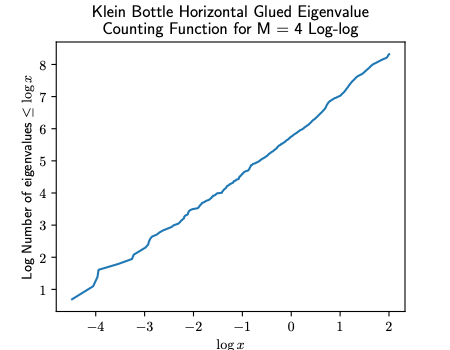}
	}
	\let \picturewidth \relax
	\caption{Eigenvalue counting functions (left) and their log-log counterparts (right).}
	\label{fig:glued-eigenvalue-counting-functions}
\end{figure}
We observe that the different identification types produce qualitatively different Weyl ratios, but for large values of $t$ they are very similar because of the fact that the $m=4$ approximation loses accuracy. We also observe that the Weyl ratios are not multiplicatively periodic. This will have implications in the next section.

\begin{figure}
	\centering
	\makebox[\linewidth]{\includegraphics[width=12cm, clip=true, trim=0.75cm 0cm 1.5cm 0cm]{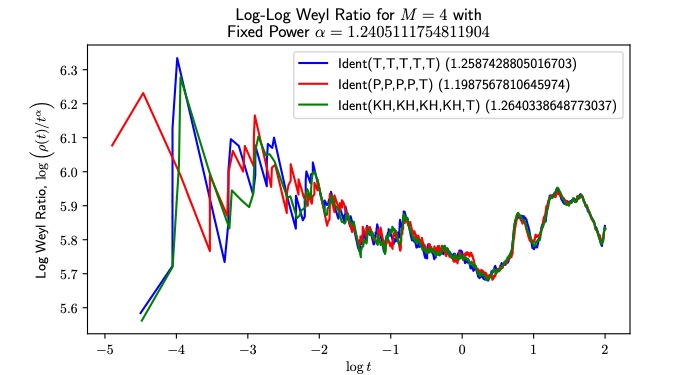}}
	\caption{Weyl plots for torus (blue), projective (red), and Klein (green) identifications.}
	\label{fig:weyl-standard-glued}
\end{figure}

In Figure~\ref{fig:sample-eigenfunctions-tg-ppg-khg} we show a sampling of graphs of eigenfunctions for $m=4$. One interesting phenomenon that we observe among levels is a miniaturization of eigenfunctions, as in Figure~\ref{fig:miniaturization-tg}. Thus every eigenfunction on level $m-1$ reappears at level $m$ repeated 8 times on each of the smaller subsquares, with the same eigenvalue, and of course this iterates. This happens for all three identifications. The general procedure is illustrated in Figure~\ref{fig:tiling-styles}.
\begin{figure}
	\newcommand{\picturewidth}{6cm}
	\centering
	\makebox[\linewidth]{
		\includegraphics[width=\picturewidth, clip=true, trim=0.5cm 1cm 0.5cm 0cm]{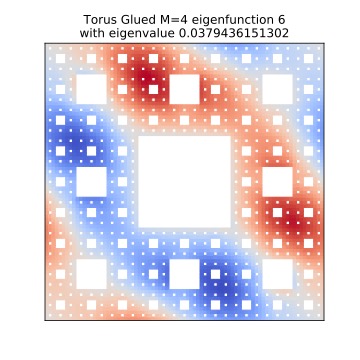}
		\includegraphics[width=\picturewidth, clip=true, trim=0.5cm 1cm 0.5cm 0cm]{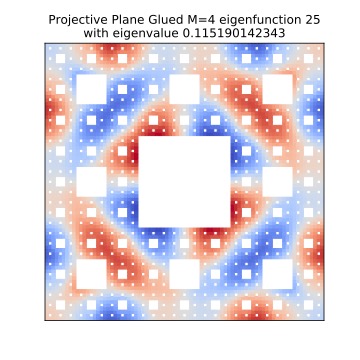}
		\includegraphics[width=\picturewidth, clip=true, trim=0.5cm 1cm 0.5cm 0cm]{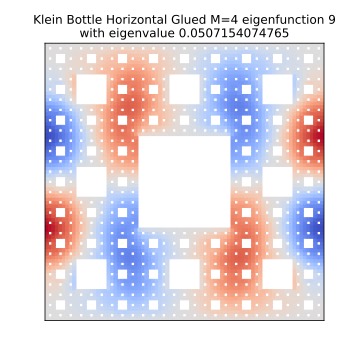}
	}
	\makebox[\linewidth]{
		\includegraphics[width=\picturewidth, clip=true, trim=0.5cm 1cm 0.5cm 0cm]{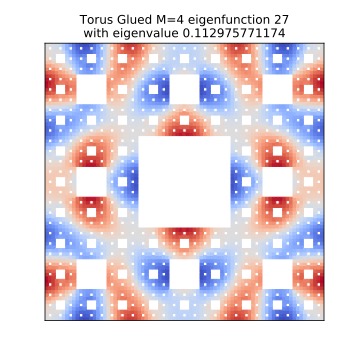}
		\includegraphics[width=\picturewidth, clip=true, trim=0.5cm 1cm 0.5cm 0cm]{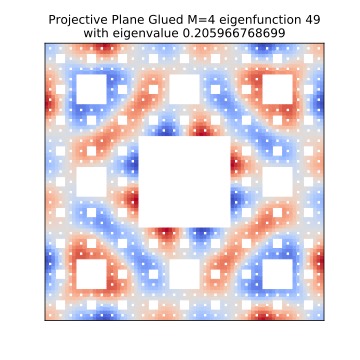}
		\includegraphics[width=\picturewidth, clip=true, trim=0.5cm 1cm 0.5cm 0cm]{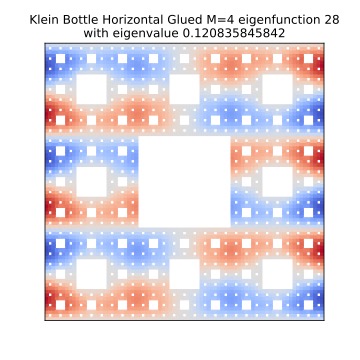}
	}
	\let \picturewidth \relax
	\caption{Sample torus glued (left), projective glued (middle), and Klein glued (right) eigenfunctions at level $m=4$.}
	\label{fig:sample-eigenfunctions-tg-ppg-khg}
\end{figure}
\begin{figure}
	\newcommand{\picturewidth}{6cm}
	\centering
	\makebox[\linewidth]{
		\includegraphics[width=\picturewidth, clip=true, trim=0.5cm 1cm 0.5cm 0cm]{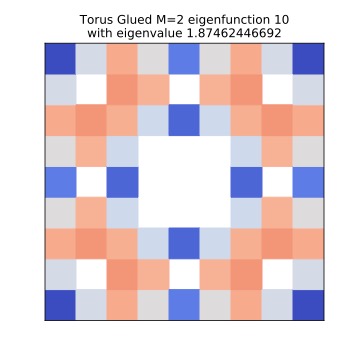}
		\includegraphics[width=\picturewidth, clip=true, trim=0.5cm 1cm 0.5cm 0cm]{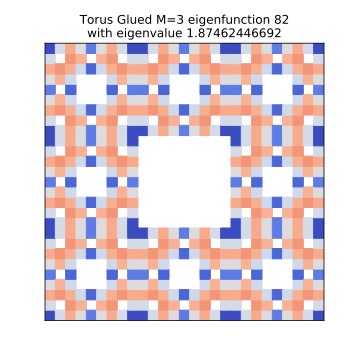}
		\includegraphics[width=\picturewidth, clip=true, trim=0.5cm 1cm 0.5cm 0cm]{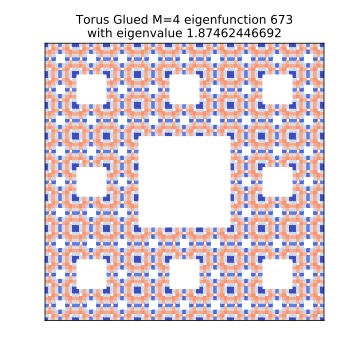}
	}
	\makebox[\linewidth]{
		\includegraphics[width=\picturewidth, clip=true, trim=0.5cm 1cm 0.5cm 0cm]{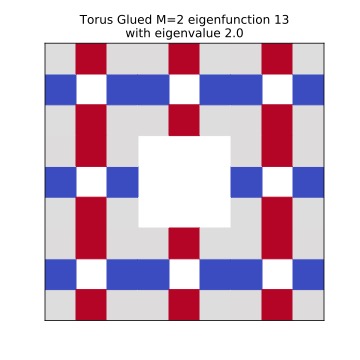}
		\includegraphics[width=\picturewidth, clip=true, trim=0.5cm 1cm 0.5cm 0cm]{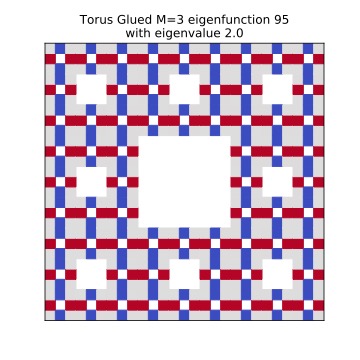}
		\includegraphics[width=\picturewidth, clip=true, trim=0.5cm 1cm 0.5cm 0cm]{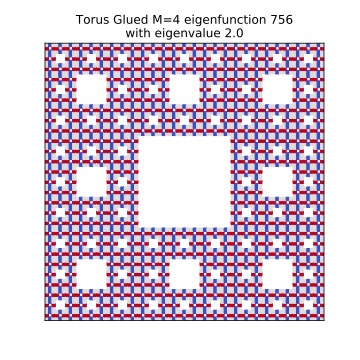}
	}
	\let \picturewidth \relax
	\caption{Two samples of miniaturization with torus glued identifications across levels $m=2,3,4$.}
	\label{fig:miniaturization-tg}
\end{figure}
\begin{figure}
	\centering
	\includegraphics[width=3cm]{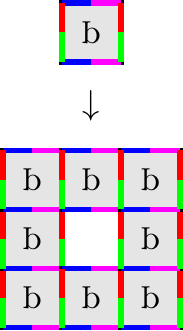}
	\hspace{1cm}
	\includegraphics[width=3cm]{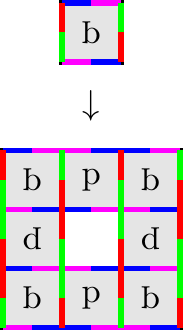}
	\hspace{1cm}
	\includegraphics[width=3cm]{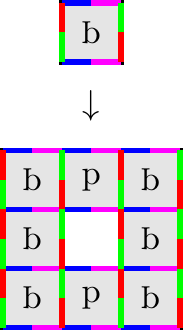}
	\caption{An eigenfunction on $V_m$ is tiled to obtain an eigenfunction on $V_{m+1}$, either for torus (left), projective (middle), or Klein horizontal (right) identifications.}
	\label{fig:tiling-styles}
\end{figure}
We can turn this observation around to attempt to describe bounded periodic eigenfunctions on $IMC$. Take any eigenfunction of $-\D^{(m)}$ on $MC_m$ and duplicate it on every copy of $MC_m$ in $IMC$. This produces an exact eigenfunction on $IMC$ that is bounded and periodic. We will use these in Section~\ref{sec: spectral resolution} to attempt to describe the spectral resolution on $IMC$.

In order to understand the relationship between the eigenfunctions of $-\D^{(m)}$ for different values of $m$ that should be regarded as refinements, we use the following reverse comparison by an averaging method. Start with an eigenfunction $u$ on $MC_m$, so $-\D^{(m)} u = \lambda^{(m)} u$. Now produce a function $\bar u$ on $MC_{m-1}$ by assigning to a cell $x$ in $MC_{m-1}$ the average value of $u$ on the eight cells in $MC_m$ that comprise $x$. Compute $-\D^{(m-1)} \bar u$ and look for a value $\lambda^{(m-1)}$ that is an eigenvalue of $-\D^{(m-1)}$ and such that $-\D^{(m-1)} \bar u \approx \lambda^{(m-1)} \bar u$. Then compare $\bar u$ to its projection $\operatorname{proj} \bar u$ on the $\lambda^{(m-1)}$-eigenspace. We consider $u$ a refinement if $\operatorname{proj} \bar u$ is close to $\bar u$. In many instances it is, as shown in Figure~\ref{fig:refinement-of-eigenfunctions-tg-and-pg}. However, this cannot always be the case, simply because there are many more eigenvalues on level $m$ than on level $m-1$. We do not see evidence of the spectral decimation property as established on the Sierpinski gasket in \cite{Fukushima-Shima}.

\begin{figure}
	\centering
	\newcommand{\picturewidth}{4.2cm}
	\makebox[\linewidth]{
		\includegraphics[width=\picturewidth, clip=true, trim=1cm 1cm 1cm 0cm]{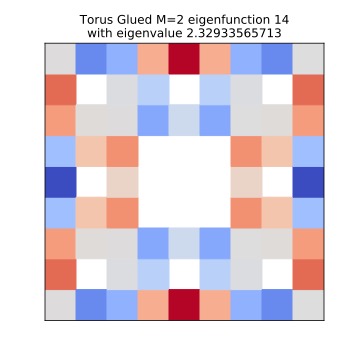}
		\includegraphics[width=\picturewidth, clip=true, trim=1cm 1cm 1cm 0cm]{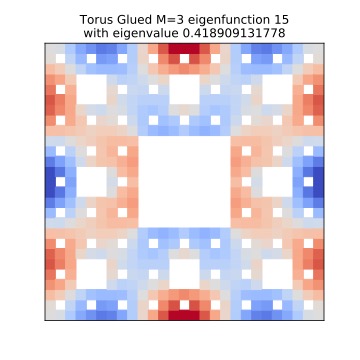}
		\includegraphics[width=\picturewidth, clip=true, trim=1cm 1cm 1cm 0cm]{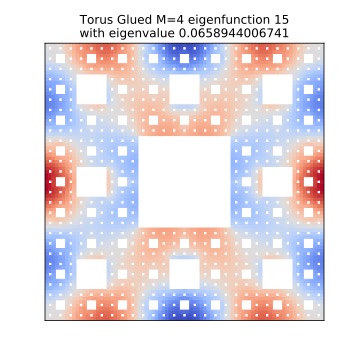}
	}
	\makebox[\linewidth]{
		\includegraphics[width=\picturewidth, clip=true, trim=1cm 1cm 1cm 0cm]{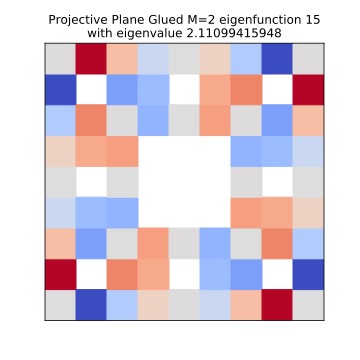}
		\includegraphics[width=\picturewidth, clip=true, trim=1cm 1cm 1cm 0cm]{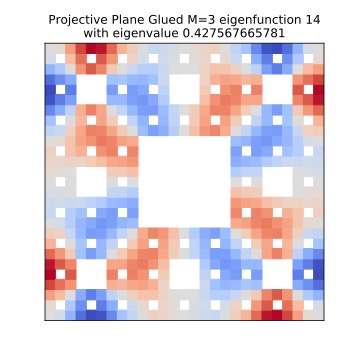}
		\includegraphics[width=\picturewidth, clip=true, trim=1cm 1cm 1cm 0cm]{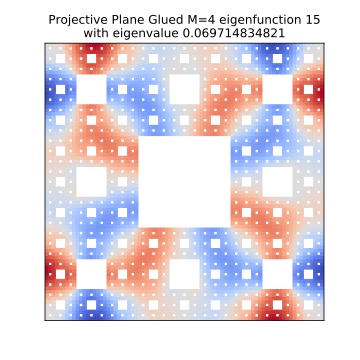}
	}
	\let \picturewidth \relax
	\caption{The top row shows an eigenfunction refining on $MCT$ from level~2~(left) to levels~3~(middle) and~4~(right, negated). The bottom row is similar, but on $MCP$.}
	\label{fig:refinement-of-eigenfunctions-tg-and-pg}
\end{figure}
\newcommand{\picturewidth}{1.1\linewidth}
\begin{figure}
	\centering
	\makebox[\linewidth]{\includegraphics[width=\picturewidth, clip=true, trim=2cm 0cm 2cm 0cm]{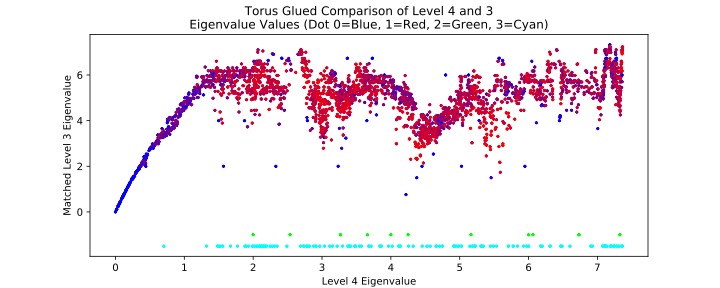}}
	\caption{Matched eigenvalues from our refinement testing on levels $m = 3, 4$ with torus glued identifications. (Projective and Klein identifications yield similar plots.) Color indicates result quality: continuously from dark blue (refines well) to red (poor match); green ($\bar u = 0$); or cyan ($\operatorname{proj} \bar u = 0$). Since green and cyan dots do not pair with level~3 eigenvalues, they are placed artificially along the bottom.}
	\label{fig:eigenfunction-refinement-eigenvalue-matching-4-3}
\end{figure}
\let \picturewidth \relax


\section{Homogeneous Identifications}
\label{sec: homogeneous}

Rather than use the same type of identification at all levels in constructing $MC_m$, we may vary the type from level to level. For example
\[ MC_4(T, P, K_H, K_V, T) \]
means we do torus identifications on the outer boundary, projective identifications on the one large vacant square, horizontal Klein bottle identifications on the eight next largest vacant squares, and so on. At the lowest level of a cell-graph approximation two cells are neighbors independent of identification type. Unlike higher level gaps, the orientation of edges does not influence the fact that the cells are neighbors. We can observe this clearly on $MC_1$, with fixed outer identification, where any identification type for the interior, vacant square yields the same cell graph. Hence the final identification type on the smallest vacant squares does not affect the spectrum at this level, but of course if we think of this as an approximation to a magic carpet fractal, this identification will play a role in the later approximations. The idea for looking at these is inspired by the work of Hambly \cite{hambly92} on Sierpinski-gasket-type fractals, and the followup in \cite{drenning-strichartz}. We call these \defnterm{homogeneous} because we use the same identification type across the board on each level. We could also consider the more general situation where every identification is allowed for each vacant square, as in \cite{hambly97, hambly00}, but we do not expect to see any structure in the spectrum with such choices.

For $m=4$ there are $4^4 = 256$ choices for identification types, although some interchanges of $K_H$ and $K_V$ will yield the same spectrum. To keep things manageable we mainly concentrate on torus and projective identifications, which reduces the total number to sixteen. In Figure~\ref{fig:weyl-TP-to-start} we show the simultaneous graphs of eight Weyl ratios for all identification types that begin with $T$ (respectively, $P$). In Figure~\ref{fig:weyl-TP-to-start-zoomed} we show a zoom of these graphs to the beginning interval~$[-5,-2]$. In Figures~\ref{fig:weyl-TP-grouping-TT}--\ref{fig:weyl-TP-grouping-PP} we show the simultaneous graphs of four Weyl ratios with the same first two identification types, again with zooms to the interval $[-5, -2]$. In Figures~\ref{fig:weyl-TP-grouping-TTT}--\ref{fig:weyl-TP-grouping-PPP} we show simultaneous graphs of two Weyl ratios with the same first three identification types, with zooms to the interval $[-5, -2]$.

\newcommand{\WeylPictureWidth}{9cm}
\newcommand{\WeylPlot}[1]{\includegraphics[width=\WeylPictureWidth, clip=true, trim=0.75cm 0cm 1.5cm 0cm]{#1}}

\begin{figure}
	\centering
	\makebox[\linewidth]{
		\WeylPlot{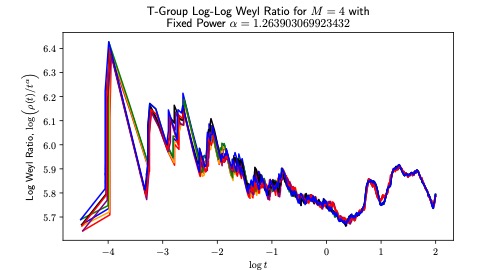}
		\WeylPlot{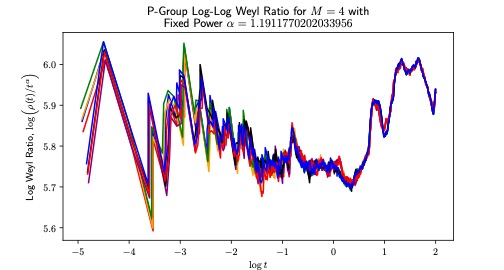}
	}
	\caption{Weyl ratios for identifications beginning with torus (left) or projective (right) identifications.}
	\label{fig:weyl-TP-to-start}

	\makebox[\linewidth]{
		\WeylPlot{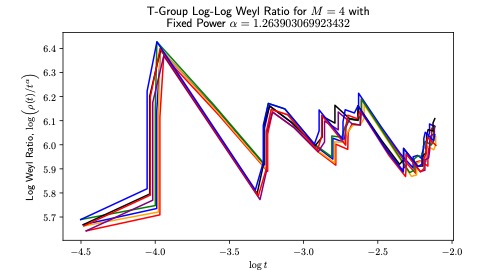}
		\WeylPlot{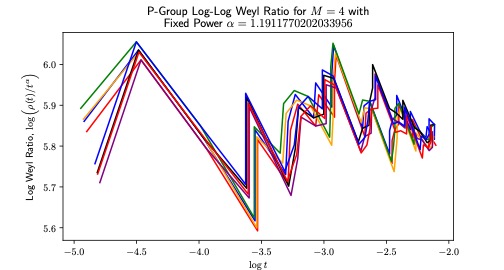}
	}
	\caption{Zoomed version of Figure~\ref{fig:weyl-TP-to-start}.}
	\label{fig:weyl-TP-to-start-zoomed}
\end{figure}


\begin{figure}
	\centering
	\makebox[\linewidth]{
		\WeylPlot{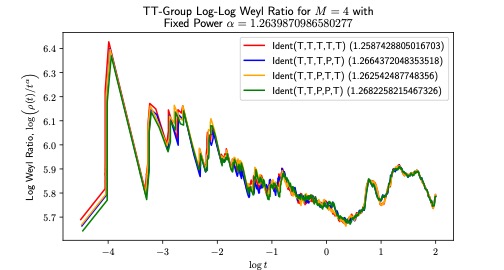}
		\WeylPlot{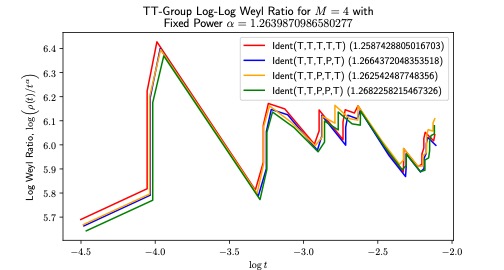}
	}
	\caption{Weyl plots for identifications beginning $T, T, \ldots$ (left) and a zoom (right).}
	\label{fig:weyl-TP-grouping-TT}
\end{figure}

\begin{figure}
	\centering
	\makebox[\linewidth]{
		\WeylPlot{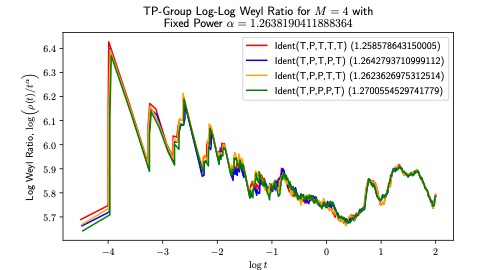}
		\WeylPlot{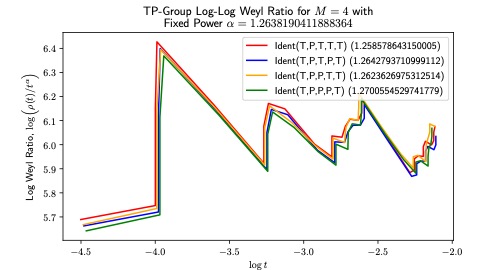}
	}
	\caption{Weyl plots for identifications beginning $T, P, \ldots$ (left) and a zoom (right).}
	\label{fig:weyl-TP-grouping-TP}
\end{figure}

\begin{figure}
	\centering
	\makebox[\linewidth]{
		\WeylPlot{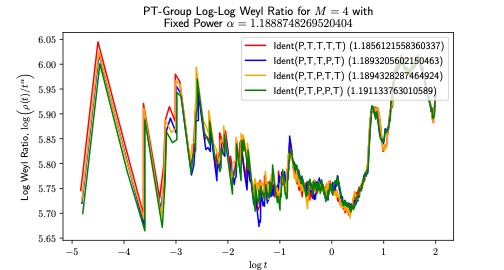}
		\WeylPlot{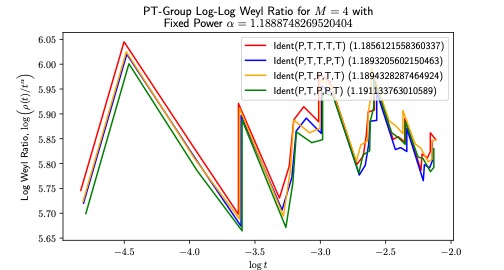}
	}
	\caption{Weyl plots for identifications beginning $P, T, \ldots$ (left) and a zoom (right).}
	\label{fig:weyl-TP-grouping-PT}
\end{figure}

\begin{figure}
	\centering
	\makebox[\linewidth]{
		\WeylPlot{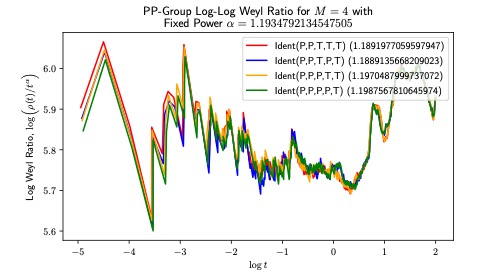}
		\WeylPlot{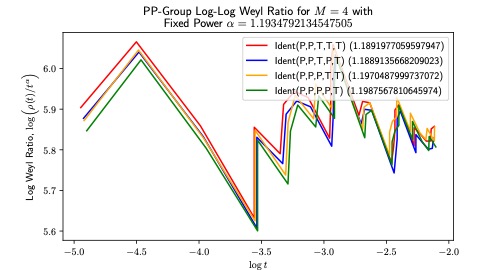}
	}
	\caption{Weyl plots for identifications beginning $P, P, \ldots$ (left) and a zoom (right).}
	\label{fig:weyl-TP-grouping-PP}
\end{figure}

\begin{figure}
	\centering
	\makebox[\linewidth]{
		\WeylPlot{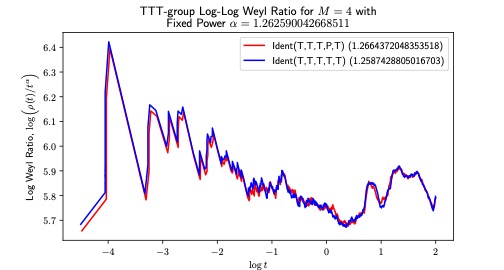}
		\WeylPlot{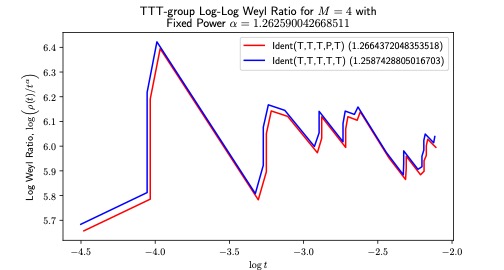}
	}
	\caption{Weyl plots for identifications beginning $T, T, T, \ldots$ (left) and a zoom (right).}
	\label{fig:weyl-TP-grouping-TTT}
\end{figure}

\begin{figure}
	\centering
	\makebox[\linewidth]{
		\WeylPlot{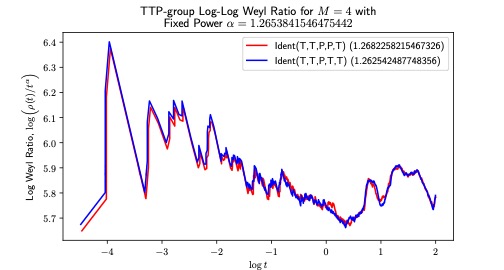}
		\WeylPlot{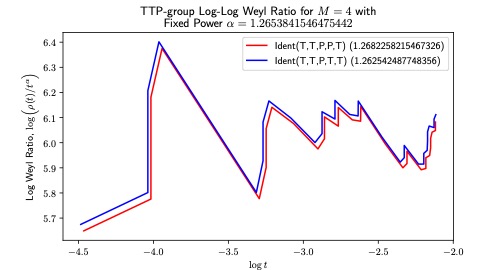}
	}
	\caption{Weyl plots for identifications beginning $T, T, P, \ldots$ (left) and a zoom (right).}
\end{figure}

\begin{figure}
	\centering
	\makebox[\linewidth]{
		\WeylPlot{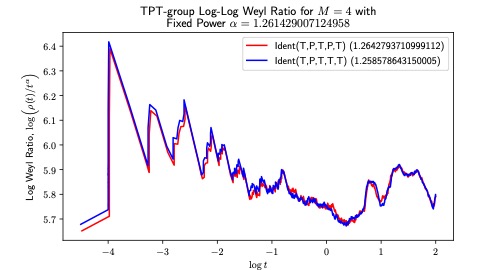}
		\WeylPlot{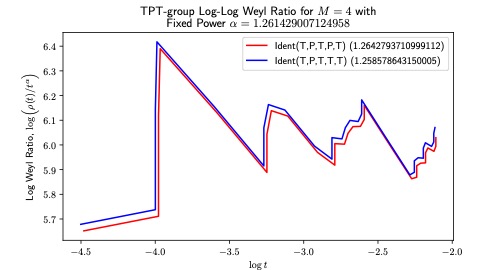}
	}
	\caption{Weyl plots for identifications beginning $T, P, T, \ldots$ (left) and a zoom (right).}
\end{figure}

\begin{figure}
	\centering
	\makebox[\linewidth]{
		\WeylPlot{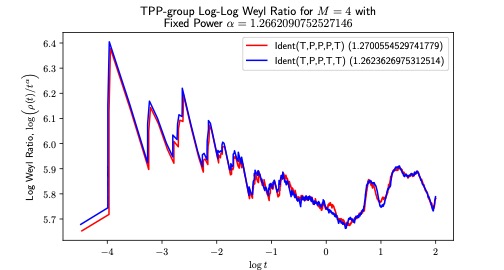}
		\WeylPlot{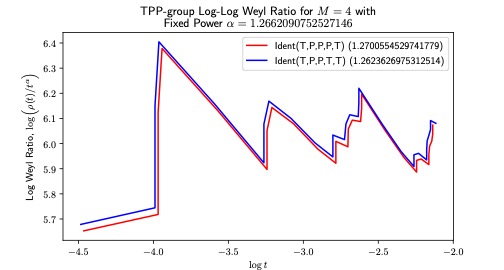}
	}
	\caption{Weyl plots for identifications beginning $T, P, P \ldots$ (left) and a zoom (right).}
\end{figure}

\begin{figure}
	\centering
	\makebox[\linewidth]{
		\WeylPlot{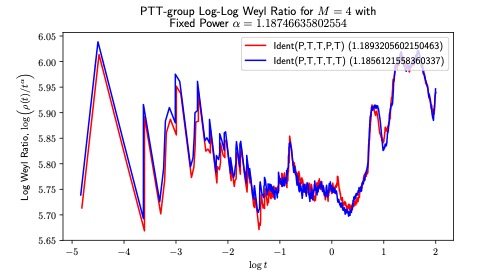}
		\WeylPlot{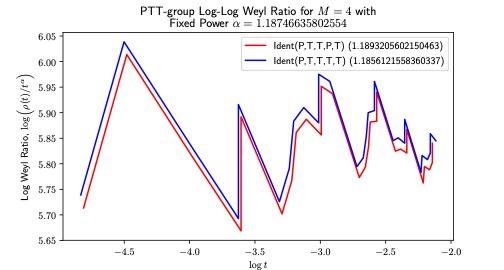}
	}
	\caption{Weyl plots for identifications beginning $P, T, T, \ldots$ (left) and a zoom (right).}
\end{figure}

\begin{figure}
	\centering
	\makebox[\linewidth]{
		\WeylPlot{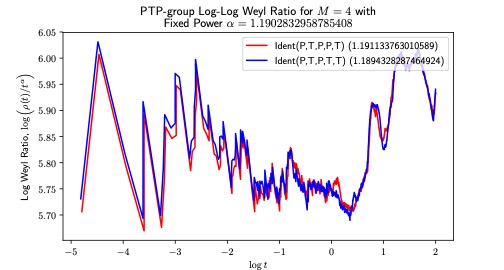}
		\WeylPlot{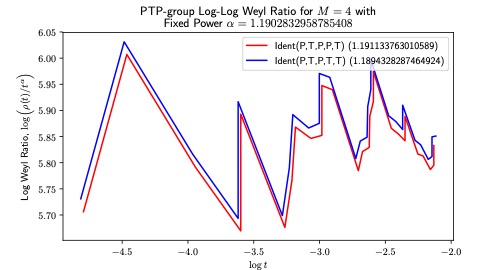}
	}
	\caption{Weyl plots for identifications beginning $P, T, P, \ldots$ (left) and a zoom (right).}
\end{figure}

\begin{figure}
	\centering
	\makebox[\linewidth]{
		\WeylPlot{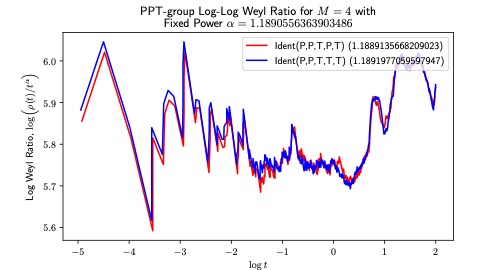}
		\WeylPlot{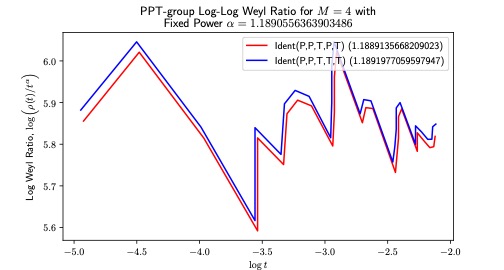}
	}
	\caption{Weyl plots for identifications beginning $P, P, T, \ldots$ (left) and a zoom (right).}
\end{figure}

\begin{figure}
	\centering
	\makebox[\linewidth]{
		\WeylPlot{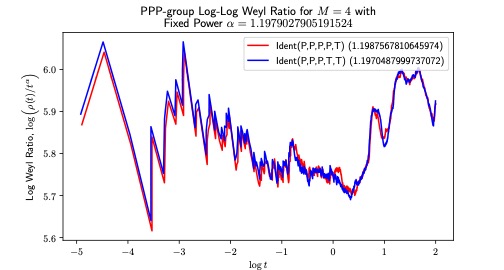}
		\WeylPlot{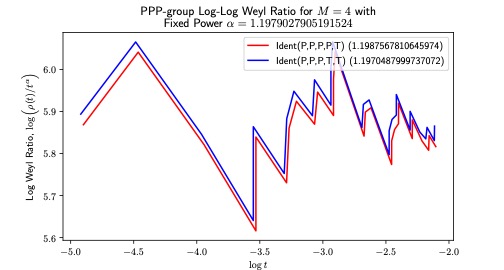}
	}
	\caption{Weyl plots for identifications beginning $P, P, P, \ldots$ (left) and a zoom (right).}
	\label{fig:weyl-TP-grouping-PPP}
\end{figure}

\let \WeylPlot \relax

A general principle called \defnterm{spectral segmentation} was introduced in \cite{drenning-strichartz} to the effect that it is possible to segment the spectrum of a fractal Laplacian so that each segment corresponds to the geometry at a certain scale. For the example studied in \cite{drenning-strichartz}, this effect, although only qualitative, was immediately apparent visually. What made the situation so clear was the fact that the Weyl ratios for the underlying Sierpinski gaskets were asymptotically multiplicatively periodic. That is not the case here, as mentioned in Section~\ref{sec: uniform}.

What the graphical evidence shown in the figures is supposed to suggest is a weak form of this principle: if two identification sequences agree in the first $k$ places, then the Weyl ratios are qualitatively the same for $\log t \le a_k$, where $a_k$ increases with $k$.


\section{The Spectral Resolution on \textit{IMC}}
\label{sec: spectral resolution}

We can use the periodic eigenfunctions discussed in Section~\ref{sec: uniform} to attempt to describe the spectral resolution on $IMC$. For each interval $[a, b)$ we need to construct a projection operator $P_{[a, b)}$ on $\ell^2(IMC)$ that is additive, and so that
\begin{align}
	\label{eq: f is lim of P f}
	f = \lim_{b \to \infty} P_{[0, b)} f \\
	\intertext{and}
	\label{eq: Df is limit of integral}
	-\D f = \int_0^\infty \lambda \, dP_{[0, \lambda)} f.
\end{align}
These identities should hold for all $f \in \ell^2(IMC)$ but it suffices to verify them for $f$ having compact support.

Let $\{ u_k^{(m)} \}$ denote an orthonormal basis of eigenfunctions on $\widetilde{MC}_m$, so
\begin{align}
	-\D^{(m)} u_k^{(m)} = \lambda_k^{(m)} u_k^{(m)}
\end{align}
and let $\tilde u_k^{(m)}$ denote the periodic extension to $IMC$ as illustrated in Figure~\ref{fig:tiling-styles}. Suppose $m$ is large enough that the support of $f$ is contained in the interior of $\widetilde{MC}_m$. Then
\begin{align}
	\label{eq: decomposition of f on MCm}
	\sum_k \langle f, \tilde u_k^{(m)} \rangle \tilde u_k^{(m)} = f \qquad \text{on } \widetilde{MC}_m ,
\end{align}
where $\langle f, \tilde u_k^{(m)} \rangle$ denotes $\sum_{x \in \supp f} f(x) \tilde u_k^{(m)} (x)$, and
\begin{align}
	\label{eq: decomposition of Df on MCm}
	\sum_k \langle f, \tilde u_k^{(m)} \rangle \lambda_k \tilde u_k^{(m)} = - \D f \qquad \text{on } \widetilde{MC}_m.
\end{align}
So we define
\begin{align}
	P_{[a, b)}^{(m)} f = \sum_{\lambda_k^{(m)} \in [a, b)} \langle f, \tilde u_k^{(m)} \rangle \tilde u_k^{(m)} ,
\end{align}
and we conjecture that the following limit exists:
\begin{align}
	P_{[a, b)} f = \lim_{m \to \infty} P_{[a, b)}^{(m)} f.
\end{align}
If so, then~\eqref{eq: f is lim of P f} follows from~\eqref{eq: decomposition of f on MCm} and~\eqref{eq: Df is limit of integral} follows from~\eqref{eq: decomposition of Df on MCm}.


\section{Acknowledgements}

E. Goodman was supported by the Haverford College Koshland Integrated Natural Sciences Center. CY Siu was supported by the Professor Charles K. Kao Research Exchange Scholarship 2015/16.

The numerical computations in this paper were done using the Python Programming Language (\url{http://www.python.org/}) \cite{python}, particularly with the packages NumPy (\url{http://www.numpy.org/}) \cite{numpy} and SciPy (\url{http://www.scipy.org/}) \cite{scipy}. Many of the figures were generated with the Matplotlib package (\url{http://www.matplotlib.org/}) \cite{Hunter:2007}.


\bibliographystyle{amsalpha}
\bibliography{bib}
\end{document}